\documentclass{article}
\usepackage[english]{babel}
\usepackage[utf8]{inputenc}
\usepackage[T1]{fontenc}
\usepackage{lmodern}
\usepackage{amsfonts}
\usepackage{stmaryrd}
\usepackage{amsmath,amsthm,amssymb}
\usepackage{a4wide}
\usepackage{array}
\usepackage{graphicx}
\usepackage{nicefrac}

\usepackage{tikz}
\newenvironment{rmq}[0]{\vspace{3mm}\noindent \textbf{Remark: }}{\vspace{3mm}}

\def\indi{\mbox{\hspace{0.2em}l\hspace{-0.55em}1}}

\newtheorem{thm}{Theorem}[section]
\newtheorem{prop}[thm]{Proposition}
\newtheorem{df}[thm]{Definition}
\newtheorem{lem}[thm]{Lemma}

\newcommand{\A}{\mathcal{A}}

\newcommand{\C}{\mathcal{C}}

\newcommand{\F}{\mathcal{F}}

\renewcommand{\O}{\mathcal{O}}
\renewcommand{\P}{\mathcal{P}}

\newcommand{\R}{\mathcal{R}}
\renewcommand{\S}{\mathcal{S}}
\newcommand{\T}{\mathcal{T}}
\newcommand{\V}{\mathcal{V}}

\newcommand{\X}{\mathcal{X}}

\newcommand{\IC}{\mathbb{C}}
\newcommand{\IN}{\mathbb{N}}
\newcommand{\IP}{\mathbb{P}}
\newcommand{\IR}{\mathbb{R}}
\newcommand{\IU}{\mathbb{U}}

\begin{document}

\title{Associative and commutative tree representations for Boolean functions\thanks{
This research was partially supported by the A.N.R. project {\em
  BOOLE}, 09BLAN0011, 
and by the P.H.C. Amadeus project {\em Probabilities and tree representations for Boolean functions}.
The Austrian authors' work has been supported by {\em FWF (Austrian Science Foundation),
National Research Area S9600, grant S9604} and {\em \"OAD, grant F03/2010}.}}

\author{\and Antoine Genitrini\thanks{Laboratoire LIP6,
CNRS UMR 7606 and Universit\'e Pierre et Marie Curie,
4 place Jussieu, 75252 Paris cedex~05, France.
Email: {antoine.genitrini@lip6.fr}.}
\and Bernhard Gittenberger\thanks{Technische Universit\"at Wien,
Wiedner Hauptstrasse 8-10/104, A-1040 Wien, Austria.
Email: {\{gittenberger, vkraus\}@dmg.tuwien.ac.at}.}
\and Veronika Kraus$^\ddag$
\and C\'ecile Mailler\thanks{Laboratoire de Math\'ematiques de Versailles,
CNRS UMR 8100 and Universit\'e de Versailles Saint-Quentin-en-Yvelines,
45 avenue des \'Etats-Unis, 78035 Versailles, France.
Email: {cecile.mailler@math.uvsq.fr}.}
}

\date{\today}

\maketitle{}

\begin{abstract}
Since the 90's, several authors have studied a probability distribution
on the set of Boolean functions on $n$ variables induced by some probability
distributions on formulas built upon the connectors $And$ and $Or$
and the literals $\{x_{1}, \bar{x}_{1}, \dots, x_{n}, \bar{x}_{n}\}$.
These formulas rely on plane binary labelled trees, known as Catalan trees.
We extend all the results, in particular the relation
between the probability and the complexity of a Boolean function, to other models of formulas:
non-binary or non-plane labelled trees (i.e. Polya trees). This
includes the natural tree class where 
associativity and commutativity of the connectors $And$ and $Or$ are
realised.
\end{abstract}

\section{Introduction.}

Consider the set $\F_n$ of Boolean functions on a set of $n$ variables.
There are $2^{2^n}$ such Boolean functions, as a value from $\{True, False\}$ can be assigned to every variable, which gives $2^n$ different assignments, for every assignment the function has output  $True$ or $False$. In this paper we consider And/Or trees, i.e. trees where internal nodes carry labels from the set $\{\land,\lor\}$ and external nodes (leaves) have labels from the set $\{x_1,\bar{x}_1,\ldots,x_n,\bar{x}_n\}$. Obviously, every such tree represents a function $f$ from $\F_n$.
We consider the uniform distribution on the set of And/Or trees of size $m$ (denoting by size the number of leaves) and are interested in the limiting probability $\IP_n(f)$ of a given function $f \in \F_n$ being computed by a random tree of size $m$, as $m$ tends to infinity, if it exists.\\

A lot of work has been going on in this field. Lefmann and Savick\'y \cite{LS97} were first to prove
the existence of the limiting probability of $f$. A survey including various numerical results was done
by Chauvin \textit{et al. }\cite{CFGG04} and Gardy~\cite{G06}. A similar study on implication trees, i.e.
Boolean trees where internal nodes carry implication labels ($\Rightarrow$) has been done by Fournier \textit{et al. }\cite{FGGG11}.\\
We base our paper on a recent work by Kozik~\cite{Kozik} on pattern languages. In his paper, Kozik proves a strong relation between the limiting probability of a given function~$f$ and its complexity~$L(f)$ (that is the minimal size of a tree computing the function $f$), asymptotically as the number of variables tends to infinity. We want to study the impact of removing step by step the restrictions on the trees, that is considering first plane, but non-binary or non-plane but binary And/Or trees, and later non-plane and non-binary trees. Considering such tree structures seems quite natural, as the new characteristics
correspond to adding the characteristics of associativity (non-binary) and commutativity (non-plane),
which are given for the $\land$ and $\lor$ operator on the level of Boolean logic.
Gardy~\cite{G06} presented some of these models and gave some numerical values of the probabilty distributions on $\F_{1}$.
In this paper we will use some recent method based on binary plane trees to obtain some results on
the probability distribution on $\F_{n}$, when $n$ is large.\\

Kozik has shown that the asymptotic order of $\IP_n(f)$ depends on $L(f)$ for binary plane trees (c.f. \cite{Kozik}). First, we compare the limiting probabilities of the constant function $True$ in the different models. %\footnote{Note that by negation, the probability distribution behaves symmetrically, i.e. the probability of $False$ will be the same as that of $True$, and just as well for all other functions and their negations.}. 
Supported by numerical results for $n$ equal to 1 and 2, we conjectured that commutativity does not matter.
Surprisingly to us, we find that both characteristics have impact on the limiting probability $\IP_n(True)$.
To be more precise, even if the order of $\IP_n(True)$ when $n$ tends to infinity is the same in all models, the asymptotic leading coefficient differs from model to model. To get more insight, we further compare probabilities of functions of complexity~$1$, those are the literals $x,\bar{x}$, in a next step.
 
Finally, we prove that for all tree models compared, the asymptotically relevant fraction of trees computing a given function $f$ is given by the set of minimal trees of $f$ expanded once in a given way, and give bounds for the arising probability distribution. A similar result is proved in~\cite{Kozik} for plane binary And/Or trees and in a paper by Fournier \emph{et al. }~\cite{FGGG11} for implication trees.

\section{Associative and commutative trees: definitions, generating functions.}\label{theta}
Kozik~\cite{Kozik} has shown that in binary plane trees the order of magnitude of the limiting probability
of a given Boolean function is related to its complexity.
We generalise this result and therefore define the complexity of a function by the following:

\begin{df}
An \emph{And/Or tree} is a labelled tree, where each internal node is labelled with one of the connectors $\{\land, \lor\}$ and each leaf with one of the literals $\{x_1,\bar{x}_1\ldots,x_n,\bar{x}_n\}$.
We define the \emph{size} of an And/Or tree to be its number of leaves.
\end{df}

\begin{df}
The \emph{complexity} $L(f)$ of a non-constant function~$f$ (i.e. $f\notin \{True,False\}$)
is given by the size of a smallest And/Or tree computing~$f$ (in the rest of the paper such trees will be called \emph{minimal} for $f$), while we define the complexity of $True$ and $False$ to be $L(True)=L(False)=0$.
\end{df}
As it will be clear later, the complexity of a function does not depend on the chosen tree model.

\begin{df}
We are considering a set $\T_{m,n}$ of And/Or trees of size $m$. 
Let $\IU_{m,n}$ be the uniform distribution on $\T_{m,n}$, and $\IP_{m,n}$ its image on the set of Boolean functions.
We call \[\IP_n=\lim_{m\to \infty}\IP_{m,n}\] the limiting distribution.
\end{df}

\begin{rmq}\label{rmq:duality}
In all models we will take into consideration, the probability of a function~$f$ is equal to the one of its negation. In fact, a tree computing $f$ can be relabelled in the following way: each connector is substituted by the other one and each literal by its negation ($x\rightarrow \bar{x}$ and $\bar{x}\rightarrow x$). The new tree we obtain belongs to the same model as $t$ and computes the function $\bar{f}$.
\end{rmq}

At first, we will present the result proven by Kozik.
This result will be generalised in the forthcoming parts of the paper.

\subsection{The classical model.}\label{theta:class}

First, let us consider the set $\T$ of binary plane trees, whose internal nodes are labelled with $\land$ or~$\lor$, and whose external nodes
are labelled with literals chosen in $\{x_1,\bar{x}_1,\ldots,x_n,\bar{x}_n\}$: each such tree computes a Boolean function on $n$ variables.
We denote by $T(z)=\sum_{m\geq 0}T_mz^m$ the generating function enumerating this set of trees\footnote{More generally, in the rest of this paper,
a generating function and its coefficients will be denoted by the same capital letter $Z(z)$ for the generating function and $Z_m$ for its coefficients.},
and by $T_f(z)$ the generating function of such trees computing the Boolean function $f$. Let us remind some well known results about this generating function:

\begin{prop}\label{prop:t(z)}
Binary And/Or trees fulfil the symbolic equation
\begin{equation}\label{eq:symbolic}
\T=\X\ |\ \T \land \T \ |\ \T \lor \T,
\end{equation}
where $\X$ is a leaf. Thus the generating function $T(z)$ verifies $T(z)= 2nz + 2T(z)^2$, and therefore, we have:
\[T(z)=\frac{1-\sqrt{1-16nz}}{4}\]
and the singularity $\rho_n$ of $T(z)$ is $\frac{1}{16n}$.
\end{prop}

Let us consider the uniform distribution on the set of trees of size $m$ and then the probability distribution $\IP_{m,n}$ it induces on the set $\F_n$ of Boolean functions on $n$ variables. The limit of this distribution when $m$ tends to infinity,
denoted by $\IP_n$ has already been studied, in particular by Lefmann and Savick\'y \cite{LS97}, Chauvin~\emph{et al.}~\cite{CFGG04} and Kozik \cite{Kozik}, who has shown the following theorem. 

\begin{thm}\cite[Kozik]{Kozik}\label{catalan-theta}
Let $f$ be a Boolean function. Then
\[\IP_n(f)\sim \frac{\lambda_f}{n^{L(f)+1}}\text{ as }n\to \infty,\]
where $L(f)$ is the complexity of $f$, i.e. the size of a minimal tree computing $f$, and $\lambda_f$
is a constant depending on $f$, which will be specified later in this paper.  
\end{thm}

\begin{rmq}
 Note that in \cite{Kozik}, the result $\IP_n(f)=\Theta\left(\frac{1}{n^{L(f)+1}}\right)$ is rigorously proven for the binary plane model, while the actual existence of the constant is suggested. In Section~\ref{sec:gen-constants} we give bounds for the constant $\lambda_f$. 
\end{rmq}

\begin{df}\label{def:essential}
A variable $x$ is essential for a function $f=f(x,x_1,\ldots,x_{n-1})$ if there exists an assignment of $True$ or $False$ to the variables $x_1,\ldots,x_{n-1}$, which we denote by $\underline{x}_0$, such that $f(True,\underline{x}_0)\neq f(False,\underline{x}_0)$. 
\end{df}

\begin{rmq}
An essential variable of $f$ appears in every tree representation of $f$.
\end{rmq}

\begin{rmq}
Note that in this theorem, $f$ (and thus $L(f)$) is fixed, and $n$ tends to infinity.
The set of essential variables of the function is finite (and does not depend on $n$).
\end{rmq}

First of all, let us define associative trees, commutative trees and then associative and commutative trees,
and the induced distributions on the set of Boolean functions $\F_n$.
%The final aim of the paper is to generalise Theorem~\ref{catalan-theta} to the latter models.

\subsection{The associative plane model.}

\begin{df}
An \emph{associative} tree is a plane tree where each node has out-degree chosen in $\IN \setminus \{1\}$.
A \emph{labelled associative} tree is an associative tree in which each external node has a label in $\{x_1,\bar{x}_1,\ldots,x_n,\bar{x}_n\}$
and each internal node has an $\land$-label or an $\lor$-label but cannot have the same label as its father.
We denote by $\A$ the family of associative trees and by $\A_{m}$ the set of such trees of size $m$.
\end{df}
Hence these trees are \emph{stratified}: the root can be labelled either by $\land$ or $\lor$ and it determines the labels of all other internal nodes.

We denote by $\IP_n^a=\lim_{m\to \infty}\IP^a_{m,n}$ the limiting distribution of Boolean functions induced by associative And/Or trees.
Our aim is to compare the limiting distributions $\IP^a_n$ and $\IP_n$.

The generating function enumerating associative trees is given by $A(z)=\hat{A}(z)+\check{A}(z)-2nz$,
where $\hat{A}$ (resp. $\check{A}$) is the generating function of associative trees rooted
at an $\land$-node (resp. an $\lor$-node) or is a single leaf. Note that $\hat{A}(z)=\check{A}(z)$ and,
\[\hat{A}(z)=2nz+\sum_{k\geq 2}\check{A}(z)^k=2nz+\frac{\hat{A}^2(z)}{1-\hat{A}(z)}.\]
Therefore,
\begin{equation}\label{eq:assocGF}
A(z)=\frac{1}{2}\left(1-2nz-\sqrt{1-12nz+4 n^2 z^2}\right)
\end{equation}
and its dominant singularity is
\[\alpha_n=\frac{3-2\sqrt{2}}{2n}.\]
Moreover, $A(\alpha_n)=\sqrt{2}-1$.

\begin{rmq}
Thanks to the Drmota-Lalley-Woods theorem (well presented in~\cite[Chapter 8]{FlSe}),
we can show that $P^a_{m,n}$ has indeed a limit when $m$ tends to infinity. We denote by $\hat{A}_f(z)$ (resp.~$\check{A}_f(z)$) the generating function enumerating associative trees computing~$f$, whose roots are labelled by $\land$ (resp.~$\lor$) or a literal. These generating functions satisfy the following system:
\[\left \{
\begin{array}{c}
\displaystyle{\hat{A}_f(z) = z\indi_{\{f \text{ lit.}\}}+\sum_{i=2}^{\infty}
\sum_{\substack{g_1, \dots, g_i, \\ g_1\land\dots\land g_i=f}}\check{A}_{g_1}(z)\cdots\check{A}_{g_i}(z)}\\
\displaystyle{\check{A}_f(z) = z\indi_{\{f \text{ lit.}\}}+\sum_{i=2}^{\infty}
\sum_{\substack{g_1, \dots, g_i, \\ g_1\lor\dots\lor g_i=f}}\hat{A}_{g_1}(z)\cdots\hat{A}_{g_i}(z).}
\end{array}
\right .\]
The Drmota-Lalley-Woods theorem says, roughly speaking, that generating functions satisfying a
system of functional equation have a dominant singularity of the same type. By transfer theorems
(see \cite{FlOd}) this implies similar behaviour of their coefficients and eventually the
existence of a limiting distribution. For a similar system of functional equations it was shown in
\cite[Section~3]{FGGG11} that all assumptions of the Drmota-Lalley-Woods theorem indeed hold. 
\end{rmq}

\begin{figure}[htb]
\centering
\fbox{\includegraphics[width=0.8\textwidth]{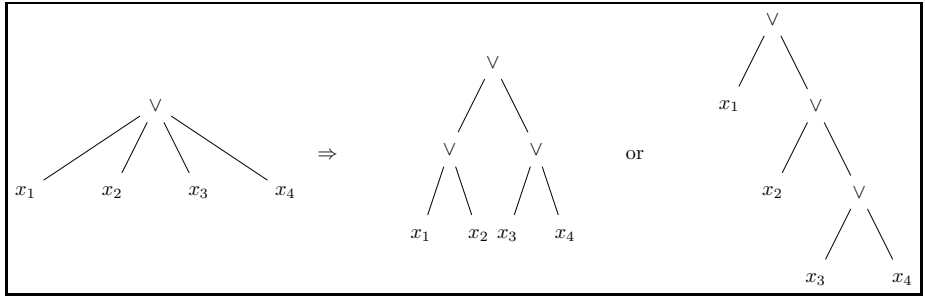}}
\caption{Two of the five possible binary trees obtained from the  associative tree.}
\label{fig:asstobinary}
\end{figure}

\subsection{The commutative binary model.}

\begin{df}
A \emph{labelled commutative} tree on $n$ variables is a non-plane binary tree where every internal node is labelled with one of the connectors $\{\land,\lor\}$ and every leaf is labelled by a literal $\{x_i,\bar{x}_i,i=1,\ldots,n\}$. We denote this family of trees by $\C$ .
\end{df}

We consider the distribution $\IP^c_{m,n}$ induced over the set of Boolean functions of
$n$~variables by the uniform distribution over such trees of size $m$.

Binary commutative trees fulfil the same symbolic equation as in the plane case (c.f. \eqref{eq:symbolic}) but because of commutativity,
the generating function of all commmutative trees on $n$ variables, counting leaves, is given implicitly by
\begin{equation}\label{eq:C(z)}
 C(z)=2nz+C(z)^2+C(z^2),
\end{equation}
where the term $\frac{1}{2}(C(z)^{2}+C(z^2))$ tracks a possible symmetry if both subtrees of the root are identical.
See Gardy~\cite{G06} for details on this model of expressions and P\'olya and Read~\cite{PR87} for more general ideas.
The system of equations for the generating functions $C_f(z)$ computing a given Boolean function~$f$ is given by
\begin{equation*}
C_f(z)=z \indi_{\{f \textrm{ lit.}\}} + \frac{1}{2}\sum_{\substack{g,h \neq f \\ g\land h=f}}C_g(z)C_h(z) +
\frac{1}{2}\sum_{\substack{g,h \neq f \\ g\lor h=f}} C_g(z)C_h(z) + C_f(z)^2+C_f(z^2).
\end{equation*}
We can prove all assumptions of the Drmota-Lalley-Woods theorem, hence we conclude that all the $(C_f(z))$ and $C(z)$ have the same singularity $\gamma_n$, and therefore $\IP^c_{m,n}$ converges to a limiting probability distribution~$P^c_n$, when $m$ tends to infinity.

\subsection{The commutative associative model.}

\begin{df}
Finally we define \emph{general labelled} trees as commutative and associative trees,
with internal nodes labelled by $\land$ or $\lor$ (with the condition that father and sons cannot have the same label),
and external nodes labelled by literals chosen in $\{x_1,\bar{x}_1,\ldots,x_n,\bar{x}_n\}$.
We denote by $\P$ this family of trees.
\end{df}
As in the other models, we consider the distribution $\IP^{a,c}_{m,n}$ induced over the set of Boolean functions
by the uniform distribution over such trees of size $m$.

Let $P(z)=\sum_{m}P_{m}z^m$ be the generating function of general trees, and $\hat{P}(z)$ (resp. $\check{P}(z)$) the generating function of general trees rooted by $\land$ (or by $\lor$, resp.) or are a leaf. We have
\begin{equation}\label{eq:polya2}P(z)=\hat{P}(z)+\check{P}(z)-2nz,\end{equation}
with
\begin{align}\label{eq:polya1}
 \left \{
\begin{array}{c}
\displaystyle{\hat{P}(z) = \exp \left(\sum_{i\geq 1}\frac{\check{P}(z^i)}{i}\right)-1-\hat{P}(z)+2nz}\\
\displaystyle{\check{P}(z) = \exp \left(\sum_{i\geq 1}\frac{\hat{P}(z^i)}{i}\right)-1-\check{P}(z)+2nz.}
\end{array}
\right.
\end{align}
Moreover, the generating functions $\hat{P}_f(z)$ and $\check{P}_f(z)$ of general trees computing $f$ satisfy the following system:
\[\left\{
\begin{array}{c}
\displaystyle{\hat{P}_f(z)=z\indi_{\{f \text{ lit.}\}}+\sum_{l=2}^{\infty}\sum_{\substack{g_1, \dots g_i, \\ g_1\land\ldots\land g_l=f}}\
\prod_{j=1}^l \left(\exp\left(\sum_{i\geq 1} \frac{\check{P}_{g_j}(z^i)}{i}\right)-1\right)}\\
\displaystyle{\check{P}_f(z)=z\indi_{\{f \text{ lit.}\}}+\sum_{l=2}^{\infty}\sum_{\substack{g_1, \dots g_i, \\ g_1\land\ldots\land g_l=f}}\
\prod_{j=1}^l \left(\exp\left(\sum_{i\geq 1} \frac{\hat{P}_{g_j}(z^i)}{i}\right)-1\right).}
\end{array}
\right.\]
Thus, we can check the hypothesis of the Drmota-Lalley-Woods theorem and conclude that the limiting distribution
$\IP^{a,c}_n$ of $\IP^{a,c}_{m,n}$, when $m$ tends to infinity, exists, and moreover,
that all the $\hat{P}_f,\check{P}_f$, $\hat{P}$ and $\check{P}$ have
the same singularity, denoted by $\delta_n$.\\

In the next parts of the paper, we will show that Theorem~\ref{catalan-theta} still holds in the associative or commutative cases. 

First, we show in Section~\ref{sec:tautologies} that the limiting ratio of tautologies is of order $\frac1{n}$, 
we compute explicitly the limit of $\IP_n(True)$ when $n$ tends to infinity for the different models. If these limits were the same, we could not conclude anything, but in fact they are all different,
which permits us to conclude that asymptotically, when $n$ tends to infinity,
the probability distributions induced by the various models are all different.
In Section~\ref{sec:literals}, we extend our results to the limiting probabilities of functions which are literals. In all models, the asymptotic ratio is of order $\frac1{n^2}$ when $n$ tends to infinity,
but the limiting ratios are different from one model to the other.
Finally, we generalise Theorem~\ref{catalan-theta} in Section~\ref{sec:gen-constants}.

\section{Limiting ratio of tautologies.}\label{sec:tautologies}
In this section we compute the limiting probability of the constant function $True$.
We recall that trees computing the function $True$ are called \emph{tautologies}.

\begin{df}
In a tree, if the path from the root to a leaf crosses only $\lor$- nodes, then
this path will be called a \emph{$\lor$-only-path}. We extend the definition to the case
such that the leaf is equal to the root (i.e. the tree has size~1).
\end{df}

As suggested by Kozik's results, the limiting probability of tautologies reduces
to the limiting probability of so-called \emph{simple tautologies}, defined by the following:

\begin{df}
A \emph{simple tautology realised by $x_i$}, $i=1\ldots n$, is a Boolean expression
which has the shape $x_i\lor \bar{x}_i \lor f$ for some Boolean function $f$,
i.e. there exists a leaf labelled by $x_i$ and a leaf labelled by $\bar{x}_i$,
both connected to the root by a "$\lor$-only-path" (c.f. Figure~\ref{fig1}).
A \emph{simple tautology} is a simple tautology realised by any literal $x\in \{x_1,\ldots,x_n\}$.
We denote by $ST_m$ the number of simple tautologies of size $m$ (on $n$ variables, $n$ is omitted for simplicity),
and $ST=\cup_{m}ST_{m}$.
\end{df} 

\begin{figure}[htb]
\begin{center}
\includegraphics[width=0.8\textwidth]{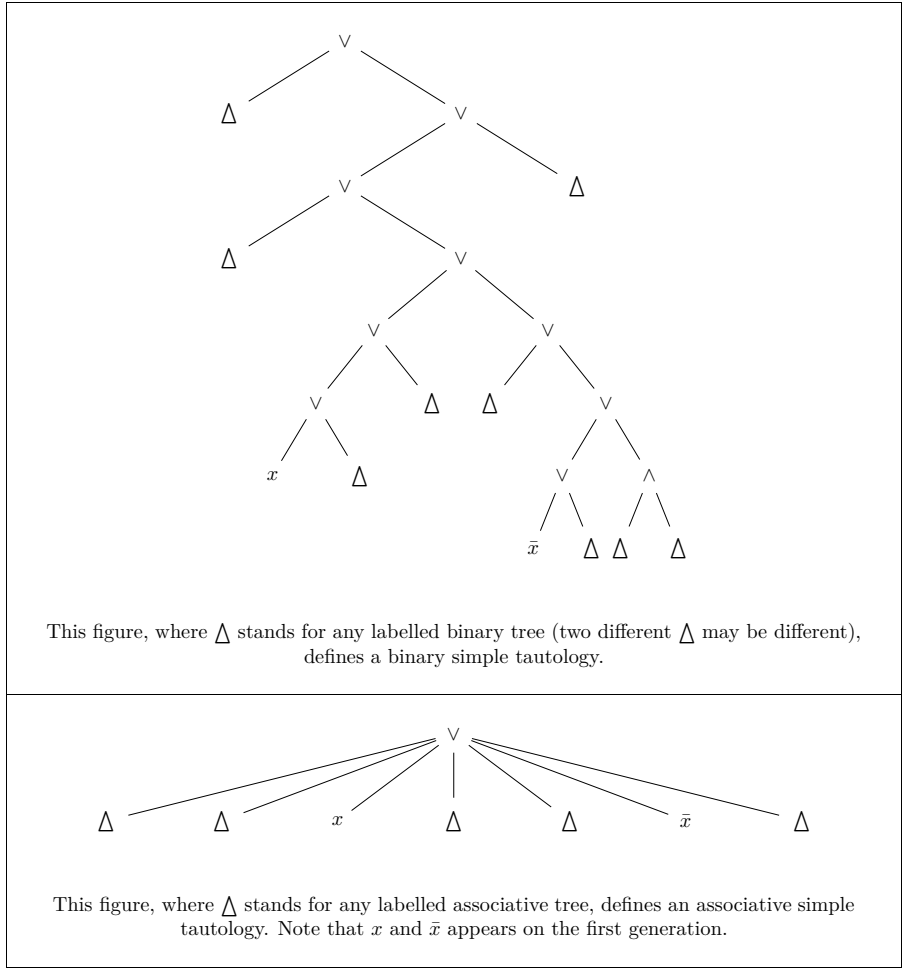}
\end{center}
\caption{Simple tautologies.}
\label{fig1}
\end{figure}

\begin{df}
\label{df:numbers}
Let $\V$ be a set of variables and $ST_m(\V)$ be the set of simple tautologies realised
by every $x\in\V$ but not by any other variable $y\not\in\V$.{}
\begin{itemize}
\item $K_{1,m}$ is the set of simple tautologies that are realised by exactly one variable:
\[K_{1,m}=\biguplus_{i=1}^n ST_m(\{x_i\}),\]
\item $K_{2,m}$ is the set of simple tautologies that are realised by exactly two different variables:
\[K_{2,m}=\biguplus_{\substack{i,j=1 \\ i\neq j}}^n ST_m(\{x_i, x_j\}),\]
$\vdots$
\item $K_{n,m}$ is the set of simple tautologies that are realised by exactly $n$ different variables:
\[K_{n,m}=ST(\{x_{1},\dots, x_n\}).\]
\end{itemize}
\end{df}

We denote by $ST^{x}(z)$ the generating function of simple tautologies realised by $x$.
Let $G(z)=ST^{x_{1}}(z)+ST^{x_{2}}(z)+\dots+ST^{x_{n}}(z) = nST^{x}(z)$. Obviously, $\forall m\in\IN,\ K_{1,m}\leq ST_{m}\leq G_{m}$,
because some tautologies are counted several times in $G$.
We get $G_m=\#K_{1,m}+2 \#K_{2,m}+\cdots +n \#K_{n,m}$.

To calculate limiting probabilities, we use the singular expansions of the considered generating
functions around their dominant singularities.
Consider the generating function $T(z)$ of a given family of And/Or trees
together with the generating function $S(z)$ of a subset $\S$ of such trees.

\begin{lem}\label{lem:limprob}
We assume that $T(z)$ and $S(z)$ have the same dominant singularity $\rho$ and a square root singular expansion
\[T(z)=a_T-b_T\sqrt{1-\frac{z}{\rho}} +\O\left(1-\frac{z}{\rho}\right), \quad
S(z)=a_S-b_S\sqrt{1-\frac{z}{\rho}}+\O\left(1-\frac{z}{\rho}\right),\]
around $\rho$. Then
\[\lim_{m\rightarrow \infty}\frac{S_m}{T_m}=\lim_{z\rightarrow \rho}\frac{S'(z)}{T'(z)}.\]
We call this number, when it exists, the \emph{limiting ratio} of the set $\S$ counted by $S(z)$.
\end{lem}

\begin{proof}
If $m$ tends to infinity, transfer lemmas (c.f.~\cite{FlSe}) give
\begin{align*}
\left.
\begin{array}{c}
\displaystyle S_m\sim\frac{b_S}{\Gamma(\frac12)}n^{-\frac32}\rho^{-m}\\
\displaystyle T_m\sim\frac{b_T}{\Gamma(\frac12)}n^{-\frac32}\rho^{-m}
\end{array}
\right\}
\Rightarrow \frac{S_m}{T_m}\sim \frac{b_S}{b_T}.
\end{align*}
Derivation of the singular expansions gives
\begin{align*}
S'(z)&\sim \frac{b_S}{2}\left(1-\frac{z}{\rho}\right)^{-\frac12},\\
T'(z)&\sim \frac{b_T}{2}\left(1-\frac{z}{\rho}\right)^{-\frac12}.
\end{align*}
Hence the result follows.
\end{proof}

\begin{rmq}
If $\S$ is the set of trees computing a given function~$f$, then,
the limiting probability of~$f$ is equal to the limiting ratio of $\S$ because for all~$m\geq 1$,
\[\IP_{m,n}(f)=\frac{\#\text{ trees of size }m\text{ computing }f}{\#\text{ all trees of size }m}=\frac{S_m}{T_m}.\]
\end{rmq}

\subsection{Binary plane trees.}\label{sec:catalan-true}

In the binary plane model, Kozik has shown that asymptotically, when $n$ tends to infinity,
all tautologies are simple tautologies.
Therefore, to estimate the probability that a binary plane tree computes the function $True$,
it suffices to count simple tautologies, and furthermore, thanks to the following proposition,
simple tautologies that are realised by only one variable (i.e. the set $K_{1,m}$).

\begin{prop}\label{prop-trick}
If $n$ tends to infinity, then
\[\lim_{m\rightarrow \infty} \frac{1}{T_{m}}\sum_{k=1}^n k\#K_{k,m}=
\lim_{m\rightarrow \infty} \frac{\#K_{1,m}}{T_m}+\O\left(\frac{1}{n^2}\right).\]
\end{prop}
The proof of the proposition is deferred to the end of this section since further technical concepts are required.

\begin{thm}\label{catalan-constant}
The limiting ratio of simple tautologies, and thus the limiting ratio of tautologies in the binary plane model is 
\[\lim_{m\rightarrow \infty}\IP_{m,n}(True)=\lim_{m\rightarrow \infty}\frac{ST_m}{T_m} \left(1+\O\left(\frac{1}{n}\right)\right)=\frac{3}{4n}+\O\left(\frac{1}{n^2}\right),
\text{ when }n \text{ tends to infinity,}\]
where $T_m$ is the total number of plane binary trees and
$ST_m$ is the number of simple tautologies of size $m$ labelled with $n$ variables. 
\end{thm}

\begin{proof}
Let us compute the generating function of simple tautologies.
First, let $g_x$ be the generating function of trees containing a leaf labelled by $x$
which is connected to the root by an $\lor$-only-path (c.f. Figure~\ref{fig:gx}) and $\bar{g}_x(z)$ the generating function of trees which are not of such shape.
Hence $g_x=T-\bar{g}_x$.

\begin{figure}[htb]
\begin{center}
\fbox{\includegraphics[width=0.3\textwidth]{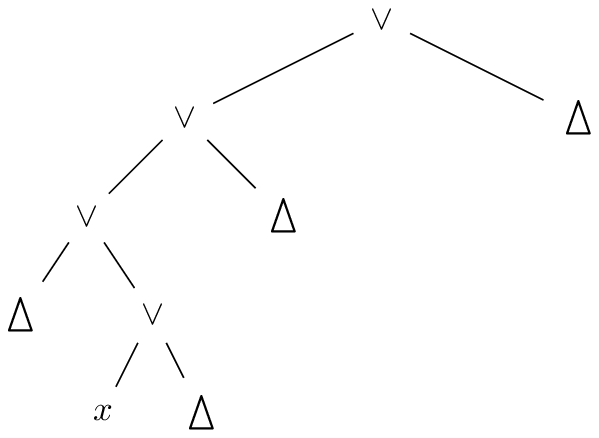}}
\end{center}
\caption{A tree counted by the generating function $g_x$.}
\label{fig:gx}
\end{figure}

The function $\bar{g}_x$ is given by:
\[\bar{g}_x(z)=T(z)^2+\bar{g}_x(z)^{2}+(2n-1)z.\]
This equation is obtained by decomposing the tree at its root: if the root is labelled by an $\land$,
the tree is not of the shape depicted in Figure~\ref{fig:gx} and both subtrees are arbitrary trees.
If the root is labelled by an $\lor$, neither of the two subtrees may have the shape of Figure~\ref{fig:gx}.
If the root is a single leaf, it must not be labelled by $x$.
By a symbolic argumentation, the three cases translate to the three terms in the equation. 
Solving this equation, using the explicit expression of $T(z)$ given by Proposition~\ref{prop:t(z)}, we get:
\[\bar{g}_x(z)=\frac{1}{2}-\frac{\sqrt{2+2\sqrt{1-16nz}-16nz+16z}}{4},\]
and thus
\begin{equation}\label{eq:catalangx}
g_x(z)=\frac{\sqrt{2+2\sqrt{1-16n}-16nz+16z}-\sqrt{1-16nz}-1}{4}.
\end{equation}

Let $h_x$ be the generating function of trees given by $t_1\lor t_2$ (or $t_2\lor t_1$),
where $t_1$ is a tree counted by $g_x$ and $t_2$ is a tree counted by $g_{\bar{x}}$,
i.e. simple tautologies realised by $x$, where $x$ and $\bar{x}$ lie in different subtrees of the root (c.f. Figure~\ref{fig:hx}).

\begin{figure}[htb]
\begin{center}
\fbox{\includegraphics[width=0.4\textwidth]{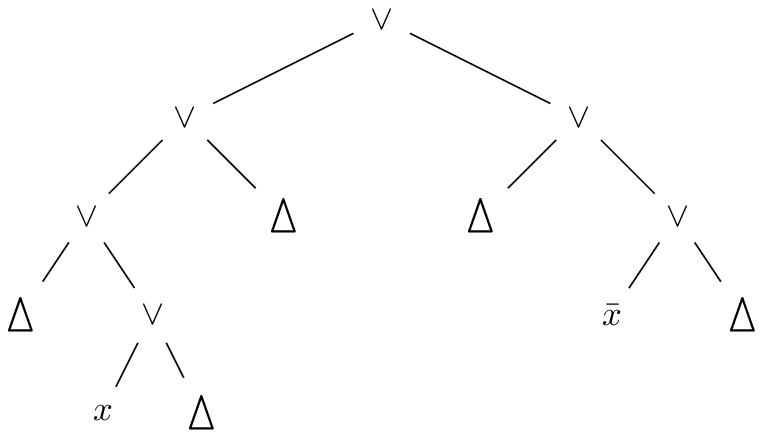}}
\end{center}
\caption{A tree counted by $h_x$.}
\label{fig:hx}
\end{figure}

Obviously, $h_x(z)=2g_x^2(z)$. Recall that $ST^{x}(z)$ is the generating function of simple tautologies
realised by the variable $x$, and $\overline{ST}^x(z)$ be the generating function of trees that are not simple tautologies realised by $x$. Again by decomposing and analysing the label of the root, we get:
\[\overline{ST}^x=T(z)^2+(\overline{ST}^x(z)^2-h_x(z))+2nz.\]
In particular, if the root is labelled by an $\lor$, neither of the
two subtrees can be a simple
tautology realised by $x$ and additionally the whole tree cannot be of the shape depicted in Figure~\ref{fig:hx}.
Solving this equation, we obtain an explicit expression for
 $\overline{ST}^x$, and $ST^{x}(z)=T(z)-\overline{ST}^x(z)$ yields an expression
 for $ST^{x}(z)$, where $Z$ denotes $Z:=\sqrt{1-16nz}$:
\begin{align*}
ST^{x}(z)&= \frac14\Big(-1-Z\Big.\nonumber\\
&+\left.\sqrt{6+6Z-2\sqrt{2+2Z-16nz+16z}-2Z\sqrt{2+2Z-16nz+16z}-48nz+16z}\right).
\end{align*}

By Proposition \ref{prop-trick},
$\lim_{m \rightarrow \infty} \frac{ST_m}{T_m} =\lim_{m\rightarrow \infty}\frac{G_m}{T_m}+ \O\left(\frac{1}{n^2}\right),$
when $n$ tends to infinity.
Due to Lemma~\ref{lem:limprob} we can compute the ratio 
\[\lim_{m\rightarrow \infty}\frac{G_m}{T_m} =\lim_{z\rightarrow \frac{1}{16n}}\frac{G'(z)}{T'(z)} = \frac{3}{4n} + \O\left(\frac{1}{n^2}\right)\]
where $G(z)=nST^{x}(z)$ is given just after Definition~\ref{df:numbers}.
Thus,
\[\lim_{m \rightarrow \infty} \frac{ST_m}{T_m}=\frac{3}{4n} +
\O\left(\frac{1}{n^2}\right).\]
Since, when $n$ tends to infinity, asymptotically almost every tautology is a simple tautology, this implies
\[\lim_{m\rightarrow \infty}\IP_{m,n}(True) = \lim_{m \rightarrow \infty} \frac{ST_m}{T_m} + \O\left(\frac{1}{n^2}\right).\]
\end{proof}

We now go back to Proposition~\ref{prop-trick}. In the following, we define pattern languages and some related vocabulary, which can be found in Kozik's paper~\cite{Kozik} for the binary case. Interpreting a given And/Or tree as an element from a pattern language, which is possible if pattern and trees have a similar structure, will lead us to the proof of Proposition~\ref{prop-trick}.

\begin{df}\label{def:pattern}
A \emph{pattern language} $\tilde{L}$ is a set of plane trees with internal nodes labelled by $\land$ or~$\lor$, and external nodes labelled by $\bullet$ or $\boxempty$. The leaves labelled by~$\boxempty$ are called \emph{placeholders} and those labelled by~$\bullet$ are called \emph{pattern leaves}. We define $s(x,y)$ as the generating function of $\tilde{L}$, with $x$ marking the pattern leaves and $y$ marking the placeholders. 

Given a pattern language $\tilde{L}$, we will denote by $L$ the set of plane labelled trees with internal nodes labelled by $\land$ or $\lor$, and external nodes labelled by literals or placeholders, such that if we replace every literal by a $\bullet$, we obtain a tree of $\tilde{L}$. Therefore, $s(2nx,y)$ is the generating function of $L$.

Given a set of trees $\T$, we define $\tilde{L}[\T]$ (resp. $L[\T]$) as the set of trees obtained by taking an element of $\tilde{L}$ (resp. $L$) and plugging an element of $\T$ in each placeholder.

Given two pattern languages $L$ and $M$, we define the composition $L[M]$ of $L$ and $M$ by the pattern language obtained by plugging $M$-patterns into the placeholders of the structures of $L$. The pattern leaves of $L[M]$ are then both the pattern leaves of $L$ and $M$.  
\end{df}

\begin{df}
A pattern language $L$ is \emph{unambiguous} if for every
family $\T$ every element of $L[\T]$ can be constructed in only one way.\\
A pattern language $L$ is \emph{subcritical} for $\T$ if the
generating function $t(z)$ of $\T$
 has a square root singularity $\rho$ and if $s(x,y)$
 is analytic in some set $\{(x,y): |x|\leq \rho + \epsilon,\: |y|\leq t(\rho)+\epsilon\}$.
\end{df}

In the following, we call a variable essential for a tree if and only if it is essential for the function computed by this tree (c.f. Definition~\ref{def:essential}).

\begin{df}
If $t$ is an element of $L[\T]$, $t$ has $q$ \emph{$L$-repetitions} if $q$ equals the difference between
the number of its $L$-pattern leaves and the number of distinct variables (and not literals) that appear in its $L$-pattern leaves.\\
Further, $t$ has $q$ \emph{$L$-restrictions} if $q$ equals the number of its $L$-repetitions plus the number of essential variables of $t$ that appear at least once in its $L$-pattern leaves.
\end{df}

For an example of repetitions and restrictions, see Figure~\ref{fig:restrictions}. 
\begin{figure}[htb]
\centering
\fbox{\includegraphics[width=0.25\textwidth]{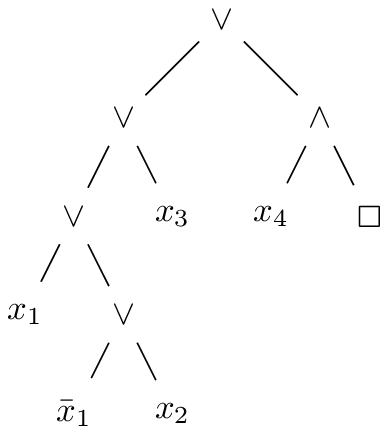}}
\caption{A binary tree with one repetition and one restriction - note
that none of the variables $x_1,x_2,x_3,x_4$ are essential as the
tree is a tautology}
\label{fig:restrictions}
\end{figure}

\begin{thm}\cite[Kozik]{Kozik}\label{thm:Jakub}
Let $L$ be a binary unambiguous language which is subcritical for
$\T$. We denote by $L[\T]_{m,n}^{[k]}$ (resp by $L[\T]_{m,n}^{[\geq k]}$) 
the number of elements of $L[\T]$ of size $m$ which have $k$
(resp. at least $k$) $L$-restrictions, and by $L[\T]_m$ the
number of elements of $L[\T]$ of size $m$. Then,
\[\lim_{m\rightarrow\infty}\frac{L[\T]_{m,n}^{[\geq k]}}{T_m} \sim
\lim_{m\rightarrow\infty}\frac{L[\T]_{m,n}^{[k]}}{T_m} \sim
\frac{d}{n^k},\]
when $n$ tends to infinity, and $d$ is a constant.
\end{thm}

Due to this theorem, we can now prove Proposition \ref{prop-trick}.
\begin{proof}[Proof of Proposition \ref{prop-trick}]
Let us consider the pattern language $S=\bullet|S\lor S| \boxempty \land \boxempty$ (c.f. \cite{Kozik}).
The set of all trees computing $True$ with exactly $i$ $S$-restrictions includes $K_i$.
Therefore, thanks to Theorem~\ref{thm:Jakub}, we get
\[\lim_{m\rightarrow \infty} \frac{\#K^m_i}{T_m}=\O\left(\frac{1}{n^i}\right),\]
when $n$ tends to infinity. Therefore, 
\[\lim_{m \rightarrow \infty}\frac{2\#K_{2,m}+\ldots+n\#K_{n,m}}{T_m}=
\O\left(\frac{1}{n^2}\right)+(n-2)\,\O\left(\frac{1}{n^3}\right)=
\O\left(\frac{1}{n^2}\right).\]
\end{proof}

\subsection{Associative plane trees.}\label{sec:assoc}

To compute the limit of $\IP^a_n(True)$ when $n$ tends to infinity, we define simple tautologies, and prove that asymptotically every tautology is a simple tautology. Therefore, we will generalise Theorem~\ref{thm:Jakub} to associative trees.

\begin{thm}
\label{thm:assoc-constant}
The limiting probability of the function $True$ in the associative model,
$\IP^a_{m,n}(True)$, is given by
\[\lim_{m\rightarrow \infty}\IP^a_{m,n}(True)=
\frac{51-36\sqrt{2}}{n} + \O\left(\frac{1}{n^2}\right).\]
\end{thm}

Let us first show that Theorem~\ref{thm:Jakub} can be generalised to the associative case, and then use
it to show Theorem~\ref{thm:assoc-constant}.

\subsubsection{Generalisation of Kozik's theorem to associative trees.}\label{sec:assoc-generalisation}

\begin{thm}\label{thm:Jakub-assoc}
Let $L$ be an unambiguous pattern language with out-degree different from $1$, which is subcritical for $\A$. We denote by $L[\A]_{m,n}^{[k]}$  (resp by $L[\A]_{m,n}^{[\geq k]}$)
the number of elements of $L[\A]$ of size $m$ which have $k$ (resp. at least $k$) $L$-restrictions, 
and by $L[\A]_m$ the number of elements of~$L[\A]$ of size $m$. Then,
\[\lim_{m\rightarrow\infty}\frac{L[\A]_{m,n}^{[\geq k]}}{A_m} \sim \lim_{m\rightarrow\infty}\frac{L[\A]_{m,n}^{[k]}}{A_m} \sim \frac{d}{n^k},\]
when $n$ tends to infinity, and $d$ is a constant.
\end{thm}
The proof of the generalisation works analogously to the one of Theorem~\ref{thm:Jakub},
still we will state the main ideas as they will be useful in the following.

Let $\tilde{\A}$ be the family of associative trees with leaves unlabelled, and let $t \in \tilde{L}[\tilde{\A}]_m$ with $l$ $L$-pattern leaves. Further, we fix the set $\V$ of essential variables and denote by $v$ the cardinality of this set, $v=|\V|$. For any $r\leq k$, the number of different
 leaf-labellings of $t$ which give $r$ $L$-repetitions and $k$ $L$-restrictions is:
\[\left\{\begin{matrix}l\\l-r\end{matrix}\right\} \begin{pmatrix}v\\k-r\end{pmatrix}
 (l-r)^{\underline{k-r}}\, (n-v)^{\underline{l-r-(k-r)}} \,n^{m-l}
 2^m,\]
where $x^{\underline{y}}=x(x-1)\ldots(x-y+1)$ and 
$\left\{\begin{matrix}y\\x\end{matrix}\right\}$
are the Stirling numbers of second
kind\footnote{The Stirling number of second kind
 $\left\{\begin{matrix}x\\y\end{matrix}\right\}$ counts the number of
 ways to partition a set of $x$ elements into $y$ non-empty sets.}.
In this formula, the different factors represent, from left to right:
\begin{itemize}
\item[-] the number of partitions of the $L$-pattern leaves into $l-r$ classes (leaves in the same class will be labelled by the same variable),
\item[-] the number of different choices for the $k-r$ essential variables that appear in the $L$-pattern leaves,
\item[-] the number of different assignments of these essential variables to the $l-r$ classes of the first term,
\item[-] the number of assignments of non-essential variables to the remaining classes of the $L$-pattern leaves,
\item[-] the number of assignments of variables to the leaves that are not $L$-pattern leaves,
\item[-] the number of ways to distribute the negations.
\end{itemize}

The following proposition is immediate:
\begin{prop}\label{Jakub-labellings}
Given an associative tree $t\in \tilde{L}[\tilde{\A}]_m$
with leaves unlabelled, the number of leaf-labellings of $t$ which make it have $k$ $L$-restrictions is:
\[(n-v)^{\underline{l-k}}\,n^{m-l}\, 2^m w_{v,k}(l),\]
where $w_{v,k}(l)=\sum_{r=0}^k
\left\{\begin{matrix}l\\l-r\end{matrix}\right\}
 \begin{pmatrix}v\\k-r\end{pmatrix} (l-r)^{\underline{k-r}}$ is a polynomial in $l$.
\end{prop}

In  \cite{Kozik} the following proposition is proved for binary trees and patterns (cf. \cite[Lemma 2.7]{Kozik}), but in fact the proof does not rely on binarity and hence the proposition holds for patterns and trees of arbitrary degree. 

\begin{prop}\label{prop:forallcases}
Let $\T$ be a set of trees whose generating function $t(z)=\sum t_m z^m$ has a unique dominating singularity $\rho$ in $\R^+$ of the square root type. Let $\tilde{L}$ be an unambiguous pattern language, subcritical for $\T$. Let $\tilde{L}[\T](m,l)$ denote the number of trees from $\tilde{L}[\T]$ of size $m$
with exactly $l$ pattern leaves. Finally, let $w(l)$ be a non zero polynolmial of degree $\delta$. Then,
\[\lim_{n\rightarrow \infty} \frac{\sum_{l \geq 0}\tilde{L}[\T](m,l)w(l)}{t_m}=c_w,\]
for some non-negative real $c_w$.

Moreover, if $w(l)$ has non-negative values and is positive at some point $l_0$, and if $L$ contains a pattern with $l_0$ non pattern leaves and at least one placeholder, then $c_w\neq 0$.
\end{prop}

Thanks to those propositions, we can now prove Theorem \ref{thm:Jakub-assoc} to associative trees:
\begin{proof}[Proof of Theorem~\ref{thm:Jakub-assoc}]
Let $L$ be an associative pattern and $\tilde{A}$ the family 
of trees from $\A$ with leaves unlabelled. 
We have, thanks to Proposition \ref{Jakub-labellings}:
\begin{equation*}
\frac{L[\T]_{m,n}^{[k]}}{A_m}=\frac{2^m \sum_{l \geq 0}\tilde{L}[\tilde{\T}](m,l)w_{k,v}(l)(n-v)^{\underline{l-k}}n^{m-l}}{A_m};
\end{equation*}
and this implies:
\begin{equation*}
\frac{L[\T]_{m,n}^{[k]}}{A_m}\leq \frac{2^m \sum_{l\geq 0}\tilde{L}[\tilde{\T}](m,l)w_{k,v}(l)n^{l-k}n^{m-l}}{(2n)^m\tilde{A}_m}.
\end{equation*}
Thanks to Proposition \ref{prop:forallcases}, we get:
\begin{equation*}
\lim_{m\rightarrow \infty}\frac{L[\T]_{m,n}^{[k]}}{A_{m}}\leq \lim_{m\rightarrow \infty}\frac{2^m \sum_{l\geq 0}L[\T](m,l)w_{k,v}(l)n^{m-k}}{(2n)^m\tilde{A}_{m}} \sim \frac{c_{k,v}}{n^k},
\end{equation*}
when $n$ tends to infinity. Moreover, we can check that $c_{k,v}$ is positive.
A lower bound can be proven analogously, the proof for the binary case is given in \cite{Kozik}. It follows that
\[\lim_{m\rightarrow\infty}\frac{L[\T]_{m,n}^{[k]}}{A_{m}} \sim
\frac{d}{n^k},\]
when $n$ tends to infinity. Moreover, we can see that:
\begin{equation*}
\frac{L[\T]_{m,n}^{[\geq k]}}{A_{m}}\leq \frac{2^m \sum_{l\geq 0}\tilde{L}[\tilde{\T}](m,l)w_{k,v}(l)n^{m-k}}{A_{m}},
\end{equation*}
and since \[\lim_{m\rightarrow\infty}\frac{L[\T]_{m,n}^{[k]}}{A_{m}} \leq \lim_{m\rightarrow\infty}\frac{L[\T]_{m,n}^{[\geq k]}}{A_{m}},\]
the theorem is proven.
\end{proof}

\subsubsection{Associative tautologies.}

\begin{prop}\label{st-assoc}
In the associative model, asymptotically when $n$ tends to infinity,
almost all tautologies are simple tautologies.
\end{prop}

The proof is very similar to the proof of the binary case (see \cite{Kozik}).
 First we need to introduce patterns:

\begin{equation}\label{eq:Rpattern}
\left\{
\begin{array}{l}
\hat{N}=\bullet | \check{N}\land \boxempty | \check{N}\land \boxempty
\land \boxempty | \ldots \\
\check{N}=\bullet | \hat{N}\lor \hat{N} | \hat{N}\lor \hat{N} \lor
\hat{N} | \ldots \\
R= \{\hat{N},\check{N}\},
\end{array}
\right.
\end{equation}
where $R= \{\hat{N},\check{N}\}$ means we start with either an $\lor$-node or an $\land$-node, and use the according pattern $\check{N}$ or $\hat{N}$, and then use both partial patterns alternatingly until the process finishes. Then~$R$ is an unambigous pattern language. 

\begin{lem}\label{assoc-subcritical}
The pattern $R$ is subcritical for associative trees.
\end{lem}

\begin{proof}
The generating function $p(x,y)$ of the labelled pattern $R$ is given by
\[p(x,y)=\hat{p}(x,y)+\check{p}(x,y)-2nx,\]
where $\hat{p}(x,y)$ (resp. $\check{p}(x,y)$) is the generating
function of the partial labelled patterns $\hat{N}$
(resp. $\check{N}$). 
These two generating functions satisfy the following system:
\begin{equation*}
\left\{
\begin{array}{l}
\check{p}(x,y)=2nx+\frac{\hat{p}(x,y)^2}{1-\hat{p}(x,y)}\\
\hat{p}(x,y)=2nx+\frac{y}{1-y}\check{p}(x,y).
\end{array}
\right.
\end{equation*}

Solving this system, we get
\begin{equation*}
\hat{p}(x,y)=\frac{1}{2}\left(2nx-y-1-\sqrt{(2nx-y-1)^2-8nx}\right).
\end{equation*}
Recall that (cf. \eqref{eq:assocGF})
\[A(z)=\frac{1}{2}\left(1-2nz-\sqrt{1-12nz+4 n^2 z^2}\right),\]
\[A(\alpha_n)=\sqrt{2}-1 \text{ and } \alpha_n=\frac{3-2\sqrt{2}}{2n}.\]
To prove that $\alpha_n$ is the dominant singularity of $p(z,A(z))$,
 it is enough to prove that it is the dominant singularity of 
$\hat{p}(z,A(z))$ and $\check{p}(z,A(z))$. Actually,
 $\hat{p}(x,y)$ and $\check{p}(x,y)$ are analytic in 
$\IC^2\setminus \{(x,y) \:|\: (2nx-1-y)^2=8nx\}$. For big enough~$n$,
 $(2n\alpha_n-\sqrt{2})^2=9(3-2\sqrt{2}) > 8n\alpha_n=4(3-2\sqrt{2})$, 
and due to non-negative coefficients the inequality holds for all
 $z\in \IR, 0\leq z\leq \alpha_n$. Thus the dominant singularity
 of $\hat{p}(z,A(z))$ is $\alpha_n$ and $R$ is subcritical for associative trees.
\end{proof}

\begin{rmq}
The $R$-pattern has an interesting property: if we set all the $R$-pattern leaves of a tree to $False$, then, the whole tree itself computes $False$. This can be checked by induction on the size of the
tree. If the pattern is only a leaf, it returns $False$. If the root is an $\lor$-node, then all subtrees of the root are patterns returning $False$ by the induction hypothesis. If the root is an $\land$-node, the leftmost subtree is a pattern returning $False$ by the induction hypothesis. Thus the whole tree computes $False$ in all cases. This property is the key point of the following proof. 
\end{rmq}

\begin{rmq}
The pattern $R$ is a generalisation of the pattern $N=\bullet|N\lor N|N \land \boxempty$, defined in \cite{Kozik} to handle the proof in the binary plane case. Note that $R[\A]=\A$, and we can find the unique element from $L[\A]$ which corresponds to a tree $A\in\A$ by starting at the root of $A$ and finding the pattern leaves by traversing the tree top-to-bottom.
\end{rmq}

\begin{proof}[Proof of Proposition \ref{st-assoc}]
Let us consider $t$ a tautology with exactly one $R[R]$-restriction (cf. Definition~\ref{def:pattern}). This restriction has to be a repetition, since a tautology does not contain essential variables. 

If the repetition is of the kind $x/x$, then we can assign all the $R$-pattern leaves to $False$, and with this assignment the whole tree computes $False$, which is impossible.

Thus the repetition has to be an $x/\bar{x}$ repetition. Let us first assume that the repetition does not appear among the $R$-pattern leaves. Thus we can assign all those leaves to $False$, and then the whole term computes $False$, which is impossible. Hence, the repetition must occur in the
 $R$-pattern leaves. Let us assume that there is a node $\nu$ labelled by $\land$ on one of the paths from the leaves labelled by $x$ and $\bar{x}$ to the root of the tree.
Then, the subtree rooted at $\nu$ has shape $t_1\land t_2 \land \ldots \land t_s$ with $s\geq 2$. Let us assume that $x$ (or $\bar{x}$) appears in $t_j$. Then, we can assign all the $R[R]$-pattern leaves of the other subtrees $(t_i)_{i\neq j}$, and all the $R[R]$-pattern leaves of the whole tree except
 those in the subtree rooted at $\nu$ to $False$. 
This makes the whole tree compute $False$, which is impossible.

Thus, $x$ and $\bar{x}$ are linked to the root by an $\lor$-only-path. As the trees are stratified,
the only possibility for $t$ is to be a simple tautology.
Thus every tree with exactly one $R[R]$-restriction computing $True$ is a simple tautology.

Moreover, there are no trees computing $True$ without $R[R]$-restrictions,
and the number of trees computing $True$ with at least two $R[R]$-restrictions
is negligible in comparison to the number of simple tautologies by Theorem~\ref{thm:Jakub-assoc} which can be applied thanks to Lemma~\ref{assoc-subcritical}.
\end{proof}

We are now able to prove Theorem \ref{thm:assoc-constant} by counting associative simple tautologies.
\begin{proof}[Proof of Theorem \ref{thm:assoc-constant}]
Let $\widetilde{ST}^{x}$ be the generating function counting the number of simple tautologies
realised by $x$ and such that $x$ and $\bar{x}$ appear only once in the first generation (i.e. at depth~1).
Then,
\[\widetilde{ST}^{x}(z)=\sum_{l\geq 2}l(l-1)(A(z)-2z)^{l-2}.\]
If $x$ or $\bar{x}$ appear at least twice in the first generation, the tree has at least two $R[R]$-restrictions, and the set of such trees is negligibly small compared to the set counted by $\widetilde{ST}^{x}$.
Thus, asymptotically, when $n$ tends to infinity, $\widetilde{ST}^{x}(z)$
counts the set of simple tautologies realised by $x$.

Finally, note that in view of Theorem~\ref{thm:Jakub-assoc}, the assertion of Proposition~\ref{prop-trick} extends to the associative case. So a Maple computation giving
\[\lim_{z\rightarrow \alpha_n}\frac{G'(z)}{A'(z)} \sim \frac{51-36\sqrt{2}}{n}\]
completes the proof of Theorem~\ref{thm:assoc-constant}.
\end{proof}

\subsection{The binary commutative model.}\label{sec:comm-tautology}

The generating function of binary commutative And/Or trees, $C(z)=\sum_{m}C_m z^m$, is given in \eqref{eq:C(z)}, and we denote by $\gamma_n$ the dominant positive singularity of $C(z)$. To compute $\gamma_n$ and $C(\gamma_n)$ we need to solve the singularity system (see~\cite[Chapter 2]{Drmota09} for details):
\begin{equation*}
\left\{
\begin{array}{l}
y = 2nz+y^2+\Gamma(z)\\
1 = 2y.
\end{array}
\right.
\end{equation*}
$\Gamma(z)$ is analytic for $|z|\leq \gamma_{n}$.
We obtain $C(\gamma_n)=\frac{1}{2}$ and $\gamma_n=\frac{1}{8n}-\frac{C(z^2)}{2n}$.
As $C(z)=2nz+\O(z^2)$, by inserting into the equation we can further derive
$\gamma=\frac{1}{8n}\left(1-\frac{1}{8n}\right)+\O\left(\frac{1}{n^3}\right)$.{}
As we need more terms in some of our calculations,
we do a more refined analysis with Maple and further obtain
\begin{equation}\label{eq:gamma}
\gamma_n=\frac{1}{8n}\left(1-\frac{1}{8n}+\frac{7}{256n^2}\right)+\O\left(\frac{1}{n^4}\right).
\end{equation}

\begin{thm}\label{thm:commbinary}
The limiting probability of the function $True$ in the binary commutative model,
$\IP_{m,n}^c(True)$, is given by
\[\lim_{m\to\infty}\IP_{m,n}^c(True)=\frac{641}{1024n}+\O\left(\frac{1}{n^2}\right).\]
\end{thm}

To prove the theorem we will extend the method of pattern languages of Kozik to the commutative case.
We consider binary commutative trees, together with a \textit{half-embedding},
that is certain branches of the tree will be plane and some will stay non-plane.
We use the plane pattern language known from Section~\ref{sec:catalan-true} given by
\[N=\bullet|N\lor N|N \land \boxempty.\]
As $N$ is plane, it is unambiguous for any tree family.
A tree of $N[\C]$ is a "mobile", that is, the pattern-trees consisting of internal nodes
and $\bullet$ and $\boxempty$-leaves are plane,
while the trees substituted into the $\boxempty$-nodes are still non-plane trees. 

\begin{rmq}
While in the plane cases the considered pattern was subcritical for the non-leaf labelled family of trees,
this is not the case for commutative trees anymore. Therefore, the strategy will be slightly different as before.
\end{rmq}

\subsubsection{Generalisation of Kozik's theorem to commutative trees.}\label{sec:comm-generalisation}
As mentioned before, in the plane case, the pattern $N$ we considered fulfilled $N[\T]=\T$. For commutative trees, this is not the case. The proof of Theorem~\ref{thm:Jakub} relies completely on plane structures and subcriticality which is not given anymore. Still, Theorem~\ref{thm:Jakub} can be generalised to mobile structures. We will adapt it and its proof, relying on the sketch in Section~\ref{sec:assoc-generalisation}, but we will need additional arguments. Note that definitions concerning the pattern language, such as restrictions, are valid in the mobile case.

\begin{thm}
\label{thm:Jakub-comm-bin}
Let $L$ be a labelled plane binary unambiguous pattern language with $\ell(x,y)${}
its generating function. Further assume that the coefficients $A_l(y)$,
given by 
\begin{equation}\label{eq:newcoeffs}
 \ell(x,y)=\sum_{l\geq0}\sum_{i\geq0}s_{i,l}y^ix^l=\sum_{l\geq0}A_l(y)x^l
\end{equation}
are subcritical for $C(z)$.\\
We denote by $L[\C]_{m,n}^{[k]}$ (resp. by $L[\C]_{m,n}^{[\geq k]}$)
the number of elements of $L[\C]$ of size $m$ which have $k$ (resp. at least $k$)
$L$-restrictions, and by $L[\C]_m$ the number of elements of $L[\C]$ of size $m$.\\
Then,
\[\lim_{m\rightarrow\infty}\frac{L[\C]_{m}^{[\geq k]}}{C_m}  =
\O\left(\frac1{n^k}\right)\text{ and } \lim_{m\rightarrow\infty}\frac{L[\C]_{m}^{[k]}}{C_m} =
\O\left(\frac1{n^k}\right)\]
when $n$ tends to infinity.
\end{thm}

Let $\tilde{L}$ be a plane pattern and $\C$ be the family of commutative trees.
Let $\Lambda$ be an element of $\tilde{L}[\C]$ of size $m$ with $l$ pattern leaves.
Note that the leaves of the non-plane parts are labelled while the pattern leaves
are unlabelled. Again, we fix the set $\V$ of essential variables and denote by $v$ the cardinality of this set, $v=|\V|$.
For any $r\leq k$, the number of different labellings of the pattern leaves
of $\Lambda$ which give $r$ $L$-repetitions and $k$ $L$-restrictions is given by
\[
\left\{\begin{matrix}l\\l-r\end{matrix}\right\}
\begin{pmatrix}v\\k-r\end{pmatrix} (l-r)^{\underline{k-r}}\, (n-v)^{\underline{l-r-(k-r)}}\, 2^l,\]
where, as in the plane case, $x^{\underline{y}}=x(x-1)\ldots(x-y+1)$ and $\left\{\begin{matrix}y\\x\end{matrix}\right\}$ are the Stirling numbers of second kind.
The different terms of the product again represent, from left to right:
\begin{itemize}
\item[-] the number of partitions of the $L$-pattern leaves into $l-r$ classes (leaves in the same class will be labelled by the same variable),
\item[-] the number of different choices for the $k-r$ essential variables that appear in the $L$-pattern leaves,
\item[-] the number of different assignments of these essential variables to the $l-r$ classes of the first term,
\item[-] the number of assignments of non-essential variables to the remaining classes of the $L$-pattern leaves,
\item[-] and the number of distribution of the negations.
\end{itemize}
 
Now, the following version of Proposition~\ref{Jakub-labellings} is obvious:
\begin{prop}\label{prop:commlabelings}
Given a binary mobile $\Lambda\in L[\C]$  with leaves unlabelled,
the number of leaf-labellings of $\Lambda$ which make it have $k$ $L$-restrictions satisfies
\[\sharp(\textrm{labellings})_k= (n-v)^{\underline{l-k}}\, 2^l w_{v,k}(l),\]
where $w_{v,k}(l)=\sum_{r=0}^k \left\{\begin{matrix}l\\l-r\end{matrix}\right\} \begin{pmatrix}v\\k-r\end{pmatrix} (l-r)^{\underline{k-r}}$
is a polynomial in $l$.
\end{prop}

We adapt Proposition~\ref{prop:forallcases} to our needs.
\begin{prop}\label{prop:forcommcases}
Let $L$ be an unambiguous labelled pattern language,
with $\ell(x,y)$ its generating function, and let $\T$ be a family of
leaf-labelled trees with generating function $T(z)$.
Further assume that the coefficients $A_l(y)$, given in \eqref{eq:newcoeffs}, are subcritical for $T(z)$.\\
Let $L[\T](m,l)$ be the number of trees of $L[\T]$ of size $m$
with exactly $l$ pattern leaves and $w(l)$ be a non-zero polynomial of degree $\lambda$. Then,
\[\lim_{m\rightarrow \infty} \frac{\sum_{l =0}^N{L}[\T](m,l)w(l)}{T_m}=c_w,\]
for some non-negative real $c_w$, where $N$ is some fixed integer.
\end{prop}

\begin{proof}
The generating function of the numerator $\sum_{l=0}^N{L}[\T](m,l)w(l)$ is denoted by $\ell_w(x,y)$.
Moreover, $w(l)=\sum_{j=0}^\lambda w_j l^{\underline{j}}$ is a representation of the polynomial $w$,
and $\ell_N(x,y)=\sum_{l=0}^NA_l(y)x^l$ is the truncation of $\ell(x,y)=\sum_{l\geq0}A_l(y)x^l$. 
Note that, 
\[x^j\frac{\partial^j \ell_N(x,y)}{\partial x^j}=\sum_{l=0}^Nl^{\underline{j}}A_l(y)x^{l}.\]
Thus
\[\sum_{j=0}^\lambda w_jx^j\frac{\partial^j \ell_N(x,y)}{\partial x^j} = \sum_{l=0}^N w(l)A_l(y)x^l.\]
Hence, the generating function $\ell_w(x,y)$ is a linear combination of
the derivatives of $\ell_N(x,y)$, which are finite sums of terms which are subcritical for $T(z)$.
Hence, $\ell_w(z,C(z))$ and $T(z)$ have the same radius of convergence.
By~\cite[Observation 3.3]{Kozik} every subcritical summand has a square root
expansion around the singularity, if $T(z)$ has a square root singularity, hence the type of singularity
of $\ell_w(z,C(z))$ is also of order $\nicefrac{1}{2}$ or of higher order if there is a cancellation. 

Thanks to a transfer lemma~\cite{FlOd}, we easily get
\[\frac{[z^m]\ell_w(z,C(z))}{[z^m]T(z)} \sim const,\]
when $m$ tends to infinity. Therefore,
\[\lim_{m\rightarrow \infty} \frac{\sum_{l \geq 0}{L}[\T](m,l)w(l)}{T_m}=c_w\]
for some non-negative constant $c_w$. Further $c_w$ is positive if there is no cancellation, and zero otherwise. 
\end{proof}

\begin{proof}[Proof of Theorem \ref{thm:Jakub-comm-bin}]
We have, thanks to Proposition \ref{prop:commlabelings}:
\begin{equation*}
\frac{L[\C]_{m}^{[k]}}{C_m} = \frac{\sum_{l=0}^N\tilde{L}[\C](m,l)w_{k,v}(l)(n-v)^{\underline{l-k}}\, 2^l}{C_m},
\end{equation*}
where $N=n-v+k$, because for larger $l$ the factor $(n-v)^{\underline{l-k}}$ gives $0$. This implies:
\begin{equation*}
\frac{L[\C]_{m}^{[k]}}{C_m}
  \leq \frac{\sum_{l=0}^N\tilde{L}[\C](m,l)w_{k,v}(l)n^{l-k}\,2^l}{C_m}
 = \frac{\sum_{l=0}^NL[\C](m,l)w_{k,v}(l)}
	{C_m}\cdot \frac{1}{n^k}.
\end{equation*}
because $L[\C](m,l)=(2n)^l \tilde{L}[\C](m,l)$.
And therefore, by applying Proposition~\ref{prop:forcommcases}, we get that
\[\lim_{m\to \infty} \frac{L[\C]^{[k]}_m}{C_m}\leq \frac{c_w}{n^k}.\]

The proof for $\lim_{m\to \infty} \frac{L[\C]^{[\geq k]}_m}{C_m}$ is analogous to the latter one.
\end{proof}

\subsubsection{Commutative tautologies.}\label{sec:embed}

\begin{prop}\label{prop:comm-st}
In the commutative model, asymptotically almost all tautologies are simple tautologies when $n$ tends to infinity.
\end{prop}

Before proving Proposition~\ref{prop:comm-st}, we introduce some
\emph{half-embedding} of a tree $t$ into the plane:
Start at the root and choose a left-right order of the children of the root. If the root is an $\land$-node,
proceed recursively with the root of the left subtree, the right subtree remains non-plane.
If the root is an $\lor$-node, proceed recursively with both subtrees. If doing so we meet a leaf,
it is a pattern leaf. Doing this for the whole tree $t$, we obtain an element of $N[\C]$,
where the non-plane subtrees are the structures substituted into the placeholders.
Now do the same half-embedding starting at every root of a non-plane subtree.
Thus we obtain an element of $N[N][\C]$. Note that different trees $t_1\neq t_2\in\C$ will create different patterns
$N[t_1]$ and $N[t_2]$, thus the function $\C\to N[\C]$ described above is an injection.
Of course, there are several ways to embed a tree $t$ with the above method,
so for every tree $t$ we choose an embedding such that the resulting $N[N]$-pattern has a minimal number
of $N[N]$-restrictions. We call such an embedding a minimal $N[N]$-embedding of~$t$
(of course there could be various minimal embeddings for one tree).

\begin{lem}\label{lem:false}
Let $t$ be a tree computing the function $True$. Then its minimal $[N]$-embedding has at least one restriction.
\end{lem}

\begin{proof}
Suppose $N[t]$ has no restriction and set all pattern leaves to $False$.
We proceed inductively. If $N[t]$ is just a leaf, it returns $False$.
If the root of $N[t]$ is an $\land$-node, the left subtree is a pattern and will,
by the induction hypothesis, return $False$, thus the whole tree returns $False$.
If the root of $N[t]$ is an $\lor$ node, both subtrees are patterns returning $False$
by the induction hypothesis. Thus the whole tree returns $False$.
\end{proof}

\begin{lem}
Let $t$ be a tree whose minimal $N[N]$-embedding has exactly
one $N[N]$-restriction. Then $t$ is a simple tautology.
\end{lem}
\begin{proof}
There are two cases to distinguish.
\paragraph*{First case:}  The restriction is of type $x/x$. Set all $N$-pattern leaves to $False$.
The same arguments as in the proof of Lemma~\ref{lem:false} show that $t$ returns $False$. 
\paragraph*{Second case:}  The restriction is of type $x/\bar{x}$. Then the restriction appears on the first level,
that is, in $N[t]$, as otherwise setting all $N$-pattern leaves to $False$ would lead to a tree
computing $False$ by the same arguments as before. If $t$ is not a simple tautology,
then there exists at least one node labelled with $\land$ on the path from the root to either $x$ or its negation.
Let $t_1$ be the non-plane subtree rooted at such a node. After the second $N$-embedding,
the $N[t_1]$ pattern contains no repetition as the whole tree $N[N][t]$ had only one $N[N]$-repetition,
thus it is easy to have $t_1$ contribute $False$ by setting all $N[t_1]$-pattern leaves to $False$.
Then the $\land$-node at $\nu$ gives $False$, thus $t$ does not compute the function $True$.
Hence, every tautology $t$ which has a minimal $N[N]$-embedding with a single repetition is a simple tautology. 
\end{proof}

\begin{lem}\label{lem:power}
Let $L$ be a pattern language with generating function $\ell(x,y)=\sum_{l\geq0}A_l(y)x^l$ and with $A_0(y)=0$,
and let $L^r$ be its $r$-th power for any $r\in\IN$, with
\[\ell^*(x,y)=\underbrace{\ell(x,(\ell(x,\cdots\ell(x}_{r\ \mathrm{ times}},y)\cdots)))=\sum_{l\geq 0}A^*_l(y)x^l\]
its generating function. Further let $\T$ be a family of trees with generating function $T(z)$.
Assume that, for all $l\geq0$, $A_l(y)$ is subcritical for $T(z)$. Then $A^*_l(y)$ is subcritical for $T(z)$.  
\end{lem}

\begin{proof}
First note that $A_0(y)=0$ means that every pattern in $L$ has at least one pattern leaf.
Obviously, this property still holds for $A^*_0(y)$.

We prove the statement by induction: the case $r=1$ is true by assumption.
Let us assume that the result holds for $r$, and let $\bar{s}(x,y)=\sum_{l\geq 0}\bar{A}_l(y)x^l$
be the generating function of $L^r$ with $\bar{A}_l(y)$ being subcritical for $T(z)$.
We want to show that $[x^l]s(x,\bar{s}(x,y))$ is subcritical for $T(z)$. It is sufficient to show that
$[x^\lambda]A_l(\bar{s}(x,y))$ is subcritical for $T(z)$ for all $\lambda$,
because $s(x,\bar{s}(x,y))=\sum_{l\geq 0}A_l(\bar{s}(x,y))x^l$ and $A_l(\bar{s}(x,y))$ is a power series in $x$, i.e. $[x^l]s(x,\bar{s}(x,y))=\sum_{j=0}^l [x^{l-j}]A_j(\bar{s}(x,y))$,
which is a finite sum of such coefficients.\\
\begin{align*}
[x^\lambda]A_l(\bar{s}(x,y))
&=[x^\lambda]\sum_{j\geq 0}s_{l,j}\bar{s}(x,y)^j\\
&=[x^\lambda]\sum_{j\geq 0}s_{l,j}\left(\sum_{\mu\geq 0} x^\mu \bar{A}_\mu (y)\right)^j\\
&=[x^\lambda]\sum_{j\geq 0}s_{l,j}\sum_{\mu_1,\ldots,\mu_j}x^{\sum \mu_i}\bar{A}_{\mu_1}\cdots \bar{A}_{\mu_j}\\
&=\sum_{j\geq 0}s_{l,j}\sum_{\mu_1+\ldots+\mu_j=\lambda}\bar{A}_{\mu_1}\cdots \bar{A}_{\mu_j}.
\end{align*}
As $\bar{A}_0(y)=0$, $\mu_i>0$ for $i=1,\ldots j$, and hence we have a maximum of $\lambda${}
factors in every summand, that is,
\[[x^\lambda]A_l(\bar{s}(x,y))
=\sum_{j= 0}^\lambda s_{l,j}\sum_{\substack{\mu_1, \dots, \mu_i, \\ \mu_1+\ldots+\mu_j=\lambda}}\bar{A}_{\mu_1}\cdots \bar{A}_{\mu_j}.\]
This is a finite sum of finite products of subcritical factors and hence it is subcritical for $T(z)$.  
\end{proof}

\begin{lem}\label{lem:subcritical}
Let $s(x,y)=\sum_{l\geq0}A_l(y) x^l$ be the generating function of the pattern $N$.
The functions $A_l(y)$ are subcritical for $C(z)$.
\end{lem}

\begin{proof}
Thanks to symbolic arguments and the recursive definition of $N=\bullet | N \lor N | N \land \boxempty$, we get 
\[s(x,y)=2nx + s(x,y)^2 + ys(x,y).\]
Solving this equation gives $s(x,y)=\frac12\left(1-y-\sqrt{(y-1)^2-8nx}\right)$.
We want to deduce an explicit formula for the $A_l(y)$ from this expression.
\begin{align*}
s(x,y)
&= \frac{1-y}{2}-\frac12 \sqrt{(y-1)^2} \sqrt{1-\frac{8nx}{(y-1)^2}}\\
&= \frac{1-y}{2}-\frac12 (1-y) \sum_{l\geq 0} \left(\begin{matrix}\nicefrac12 \\ l\end{matrix}\right)(-8n)^l (y-1)^{-2l} x^l,
\end{align*}
since $s(0,0)=0$. Therefore, $A_l(y)=-\frac12 (1-y)  \left(\begin{matrix}\nicefrac12 \\ l\end{matrix}\right) (-8n)^l (y-1)^{-2l}$
is a rational function in $y$ and its radius of convergence is $1$. Hence, these functions are subcritical for $C(z)$.
\end{proof}

\begin{proof}[Proof of Proposition \ref{prop:comm-st}]
Let $t$ be a tree in $\C$ which computes $True$. Then there is at least one variable~$x$
appearing twice in the leaves of $t$, because otherwise the tree cannot be a tautology (induction on the size of the tree).
We half-embed $t$ into the plane as described before.
As this $N$-embedding represents an injection it follows
that $C_m^{[k]} \leq (N[\C])_m^{[k]}$, where $C_m^{[k]}$ denotes the number of trees from $\C$ of size $m$ whose minimal half-embeddings have $k$ restrictions. Hence,
by Theorem~\ref{thm:Jakub-comm-bin}, which can be applied thanks to Lemmas~\ref{lem:power} and~\ref{lem:subcritical}:
\begin{align*}
\frac{C_m^{[k])}}{C_m}\leq\frac{N[\C]_m^{[k])}}{C_m}=\O\left(\frac{1}{n^k}\right),
\end{align*}
and thus asymptotically almost all tautologies
in a binary commutative And/Or tree are simple (and have a minimal $N[N]$-embedding with one restriction).
Proposition \ref{prop:comm-st} is thus proved.
\end{proof}

\begin{proof}[Proof of Theorem \ref{thm:commbinary}]
Let $g_x(z)$ be the generating function counting the trees in $\C$ with a $\lor$-only-path
from the root to a leaf labelled with $x$. It is given by $g_x(z)=C(z)-\bar{g}_x(z)$ with
\begin{align}\label{eq:commg}
\bar{g}_x(z)=(2n-1)z+\frac{1}{2}\left(C^2(z)+C(z^2)\right)+\frac{1}{2}\left(\bar{g}_x^2(z)+\bar{g}_x(z^2)\right),
\end{align}
because a tree rooted at an $\land$-node cannot contain an $\lor$-only-path from the root,
while if the root is labelled with $\lor$ both subtrees of the root must not contain an $\lor$-only-path to an $x$-leaf. 

The generating function $ST^{x}(z)$ which counts trees which are a simple tautology realized by $x$ is given by $ST^{x}(z)=C(z)-\overline{ST}^x(z)$, where $\overline{ST}^x(z)$ counts trees which are not simple tautologies realised by $x$. Similarly to $\overline{g}_x(z)$, such a tree is either rooted at an $\land$-node, or it is rooted at an $\lor$-node, and both subtrees of the root are not simple tautologies. Still, it could return $True$ if one of the subtrees contains an $\lor$-only-path to $x$ and the other subtree contains an $\lor$-only-path to $\overline{x}$.
This gives the following implicit equation for $\overline{ST}^x(z)$.

\begin{align}\label{eq:commG}
\overline{ST}^x(z)=2nz+\frac{1}{2}\left(C^2(z)+C(z^2)\right)+\frac{1}{2}\left((\overline{ST}^x)^2(z)+\overline{ST}^x(z^2)\right)-g_x(z)g_{\overline{x}}(z),
\end{align}
To calculate the limiting ratio of simple tautologies, we need to determine
$n(1-\lim_{z\to\gamma_n}\frac{(\overline{ST}^x)'(z)}{C'(z)})$, where the factor $n$ is the choice of $x$
in the set of variables, and we use an analogue of Proposition~\ref{prop-trick} as well as Lemma~\ref{lem:limprob}. We denote by
$u_n:=\bar{g}_x(\gamma_n)$, $v_n:=\bar{g}(\gamma_n^2), U_n:=\overline{ST}^x(\gamma_n)$ and
$V_n:=\overline{ST}^x(\gamma_n^2)$, and compute $U_n$ up to terms of order $\frac{1}{n^2}$. 
From \eqref{eq:commg} we get
\begin{align}
u_n&=(2n-1)\frac{1}{8n}\left(1-\frac{1}{8n}\right)+\frac12\left(\frac14+C(\gamma_n^2)\right)+\frac{1}{2}(u_n^2+v_n)+\O\left(\frac{1}{n^{2}}\right)\label{eq:un}\\
v_n&=(2n-1)\frac{1}{64n^2}\left(1-\frac{1}{8n}\right)^2+\frac12(C^2(\gamma_n^2)+C(\gamma_n^4))+\frac{1}{2}(v_n^2+\overline{g}_x(\gamma_n^4))+\O\left(\frac{1}{n^{2}}\right)\label{eq:vn}
\end{align}
We know that $C(z^2)=2nz^2+\O(z^4)$, hence $C(\gamma_n^2)=\frac{1}{32n}+\O(\frac{1}{n^{2}})$.
Inserting this into \eqref{eq:vn} we can compute $v_n=\frac{1}{32n}+\O\left(\frac{1}{n^2}\right)$,
and with \eqref{eq:un}, we compute $u_n=\frac12-\frac{1}{4n}+\O\left(\frac1{n^2}\right)$.
Solving the equations for $U_n$ and $V_n$, up to terms of order
$\frac{1}{n^2}$, we get
\[V_n=\frac{1}{32n}-\frac{7}{1024n^2}+\O\left(\frac{1}{n^3}\right) \quad\textrm{ and }
\quad U_n=\frac12-\frac{129}{1024n^2}+\O\left(\frac1{n^3}\right).\]
Derivating $(\overline{ST}^x)'(z)$ and $\bar{g}_x'(z)$, we obtain
\begin{align*}
\bar{g}_x'(z)&=2n-1+C(z)C'(z)+zC'(z^2)+\bar{g}_x(z)\bar{g}'_x(z)+z\bar{g}_x'(z^2),\\
(\overline{ST}^x)'(z)&=2n+C(z)C'(z)+zC'(z^2)+\overline{ST}^x(z)\bar{G}'_x(z)+z(\overline{ST}^x)'(z^2)-2g(z)g'(z).
\end{align*}
Hence 
\begin{align*} 
\lim_{z\to\gamma_n}\frac{\bar{g}_x'(z)}{C'(z)}&=\frac{1}{1-\bar{g}_x(z)}\left(\frac{2n-1}{C'(z)}+C(z)+\frac{zC'(z^2)}{C'(z)}+\frac{z\bar{g}_x'(z^2)}{C'(z)}\right)\\
&\sim\frac{1}{2(1-u_n)}=1-\frac1{2n}+\frac{1}{4n^2}+\O(\frac1{n^3}),\\
\lim_{z\to\gamma_n}\frac{(\overline{ST}^x)'(z)}{C'(z)}&=\frac{1}{1-\overline{ST}^x(z)}\left(\frac{2n}{C'(z)}+C(z)+\frac{zC'(z^2)}{C'(z)}+\frac{z(\overline{ST}^x)'(z^2)}{C'(z)}-\frac{2g(z)g'(z)}{C'(z)}\right)\\
&\sim\frac{1}{1-U_n}\left(\frac{1}{2}-2u_n\lim_{z\to\gamma_n}\frac{\overline{g}_x'(z)}{C'(z)}\right)\sim\left(2-\frac{129}{512n^2}\right)\left(\frac12-\frac1{4n^2}\right)\\
&=1-\frac{641}{1024n^2}+\O\left(\frac1{n^3}\right).
\end{align*}
The result of Theorem \ref{thm:commbinary} follows immediately. 
\end{proof}

\subsection{The associative and commutative model.}\label{sec:ass-comm-tautology}

The generating function of associative commutative And/Or trees $P(z)$ is given in \eqref{eq:polya2} and \eqref{eq:polya1}. Note that $\hat{P}(z)=\check{P}(z)$. Let $\delta_n$ be the dominant positive singularity of $\hat{P}(z)$,
and hence also of $P(z)$. To get $\delta_n, \hat{P}(\delta_n)$ and
$P(\delta_n)$ we need to solve the system
\begin{equation*}
\left\{
\begin{array}{l}
y = e^y\cdot \Pi(z)-1-y-2nz\\
1 =e^y\cdot \Pi(z)-1,
\end{array}
\right.
\end{equation*}
with $\pi(z)=\exp(\sum_{i\geq 2}\hat{P}(z^{i})/i)=1+n^2z^2+\O(z^3)$,
(since $\hat{P}(z)=2nz+\O(1/n^{2})$). Therefore $\pi(z)\sim1$ for $z=\O\left(\frac1n\right)$ and $n$ tending
to infinity, hence the second equation gives $e^{y(z)}\sim2$ or $y(z)\sim\ln(2)$.
Inserting this value into the first equation gives
$y=\frac{1+2nz}{2}$
and thus the first order asymptotic of $\delta_n$ is $\delta_n\sim\frac{2\ln2-1}{2n}$.
\begin{thm}\label{thm:asscommt}
The limiting probability of the function $True$ in the binary commutative model,
$\IP_{n}^{a,c}(True)$, is given by
\[\lim_{m\to\infty}\IP_{n,m}^{a,c}(True)=\frac{(2\ln2-1)^2}{4n}+\O\left(\frac{1}{n^2}\right).\]
\end{thm}

To prove the theorem we will again use mobiles, using the unambiguous pattern $R = \{\hat{N},\check{N}\}$ from Section~\ref{sec:assoc}, given in~\eqref{eq:Rpattern}. We can prove that its coefficients $A_l(y)$ are subcritical for $\P$.

\begin{lem}
Let $p(x,y)=\sum_{l\geq 0}A_l(y)x^l$ being the generating function
of the pattern language $R$. The functions $A_l(y)$ are subcritical for $P(z)$.
\end{lem} 

\begin{proof}
The generating function of the $R$ pattern is $p(x,y)=\hat{p}(x,y)+\check{p}(x,y)-2nx$ where 
\[\hat{p}(x,y)=\frac12\left(2nx-y-1-\sqrt{(2nx-y-1)^2-8nx}\right).\]
Thus,
\begin{align*}
p(x,y)
&= -(y+1) - \sqrt{(y+1)^2}\sqrt{1-\frac{4nx(n-1+y)}{(y+1)^2}}\\
&=-(y+1)+(y+1)\sum_{l\geq 0}\left(\begin{matrix}\nicefrac12\\ l\end{matrix}\right) (y+1)^{-2l}(-4nx)^l (n-1+y)^l,
\end{align*}
since $p(0,0)=0$. Hence, the $A_l(y)$ are rational functions with radius of convergence $1$, which is larger than $\delta_n$. Thus $A_l(y)$ is subcritical for $P(z)$.
\end{proof}

\subsubsection{Generalisation of Kozik's theorem to associative and commutative trees.}

\begin{thm}\label{thm:Jakub-ass-comm}
Let $L$ be a labelled unambiguous pattern language with out-degree different from~1. Further assume that the coefficients $A_l(y)$, given in \eqref{eq:newcoeffs}, are subcritical for $P(z)$.\\
We denote by $L[\P]_{m,n}^{[k]}$ (resp. by $L[\P]_{m,n}^{[\geq k]}$) the number of elements of $L[\P]$ of size $m$ which have $k$ (resp. at least $k$) $L$-restrictions. Then,
\[\lim_{m\rightarrow\infty}\frac{L[\P]_{m,n}^{[\geq k]}}{P_m} =
\O\left(\frac{1}{n^k}\right) \text{ and }
\lim_{m\rightarrow\infty}\frac{L[\P]_{m,n}^{[k]}}{P_m} =
\O\left(\frac{1}{n^k}\right),\]
when $n$ tends to infinity.
\end{thm}

The proof of Theorem~\ref{thm:Jakub-ass-comm} now
is an easy generalisation of Sections~\ref{sec:assoc-generalisation}
and~\ref{sec:comm-generalisation}. We use mobiles on
a associative plane pattern $L$, that is pattern leaves are
on plane paths from the root, while commutative trees have
been substituted in the $\boxempty$-nodes of the plane pattern. 
We can easily extend Proposition~\ref{prop:commlabelings}.

\begin{proof}[Proof of Theorem \ref{thm:Jakub-ass-comm}]
As in previous parts, Proposition~\ref{prop:commlabelings} gives:
\begin{equation*}
\frac{L[\P]_{m,n}^{[k]}}{P_m}=\frac{\sum_{l\in \IN}\tilde{L}[\P](m,l)w_{k,v}(l)(n-v)^{\underline{l-k}}2^l}{P_m};
\end{equation*}
which implies:
\begin{equation*}
\frac{L[\P]_{m,n}^{[k]}}{P_m}\leq \frac{\sum_{l\in \IN}\tilde{L}[\P](m,l)
w_{k,v}(l)n^{l-k}2^l}{P_m}=\frac{\sum_{l\in \IN}L[\P](m,l)w_{k,v}(l)}{P_m}.
\end{equation*}
Hence the result follows from Proposition~\ref{prop:forcommcases}.
\end{proof}

\subsubsection{Non-plane associative tautologies.}

\begin{prop}\label{prop:asscommst}
In the associative and commutative model, asymptotically almost all tautologies are simple tautologies, when $n$ tends to infinity.
\end{prop}

Again we introduce a \emph{half-embedding} of a tree $t$ into the plane:
Start at the root and choose a left to right order of the children of
the root.
If the root is an $\land$-node, proceed with the leftmost child of the
root. If the root was an $\lor$-node, then do the same for every child
of the root. If we end up at a leaf, this is a pattern leaf.
By this procedure we obtain an element of $R[\P]$.
Applying the same procedure to every root of a commutative subtree,
we obtain an element of $R[R][\P]$, we call it an $R[R]$-embedding of
$t$. There are several ways to embed $t$, choose one embedding with
a minimal number of $R[R]$-restrictions.
Again, the function $t\mapsto R[R]_{min}(t)$ represents an injection. 

Now looking at all trees with a minimal $R[R]$-embedding having
exactly one restriction, we can proceed in the same way as in the
proof of Theorem~\ref{thm:assoc-constant} to prove that they are simple tautologies.

\begin{proof}[Proof of Proposition ~\ref{prop:asscommst}]
Let $t\in\P$ be a tree that computes True.
We half-embed $t$ and argue as in Section~\ref{sec:embed} to prove
Proposition~\ref{prop:asscommst} with the help of
Theorem~\ref{thm:Jakub-ass-comm}. 
\end{proof}

\begin{proof}[Proof of Theorem \ref{thm:asscommt}]
We define $ST^{x}(z)$ as previously and obtain
\begin{align}
ST^{x}(z)&=z^2\sum_{\ell\geq0}Z_\ell((\hat{P}(z)-2z,\hat{P}(z^2)-2z^2,\ldots)\nonumber\\
&=z^2\exp\left(\sum_{\ell\geq1}\frac{\hat{P}(z^\ell)-2z^\ell}{\ell}\right),\label{eq:asscommG}
\end{align}
where $Z_\ell(s_1,s_2,\ldots)$ denotes the cycle index of the
symmetric group on $\ell$ elements (c.f.~\cite{PR87}). Hence
\begin{align*}
G'_x(z) &= 2z(\exp\left(\sum_{\ell\geq1}\frac{\hat{P}(z^\ell)-2z^\ell}{\ell}\right)+z^2
\exp\left(\sum_{\ell\geq1}\frac{\hat{P}(z^\ell)-2z^\ell}{\ell}\right)\left(\sum_{\ell\geq1}z^{\ell-1}(\hat{P}'(z^i)-2)\right)\\
&=\frac{2}{z}ST^{x}(z)+ST^{x}(z)\left(\hat{P}'(z)-2+\sum_{\ell\geq2}z^{\ell-1}(\hat{P}'(z^i)-2)\right)
\end{align*}
At $z=\delta_n \sim \frac{2\ln2-1}{2n}$, $ST^{x}(z)$ equals
\[ST^{x}(\delta_n)=\delta_n^2\underbrace{\exp\left(\sum_{i\geq1}\frac{\hat{P}(\delta_n^i)}{i}\right)}_{=2}
\underbrace{\exp\left(\sum_{i\geq1}\frac{-2\delta_n^i}{i}\right)}_{=(1-\delta_n)^2\sim1}\sim2\delta_n^2,\]
Hence, due to $P(z)=2\hat{P}(z)-2nz$,
\begin{align*}
\lim_{z\to\delta_n}\frac{G'_x(z)}{P'(z)}&=\lim_{z\to\delta_n}\frac{ST^{x}(z)\hat{P}'(z)}{P'(z)}\\
&=\lim_{z\to\delta_n}\frac{ST^{x}(z)\hat{P}'(z)}{2\hat{P}'(z)-2n}=\frac{(2\ln2-1)^2}{4n^2},
\end{align*}
and
\[\lim_{z\to\delta_n}\frac{G'(z)}{P'(z)}=n\frac{(ST^{x})'(z)}{P'(z)}\sim\frac{(2\ln2-1)^2}{4n}.\]
\end{proof}

\section{Limiting probabilities of literals.}\label{sec:literals}
In this section, we will compute the limiting probabilities of functions of complexity $L(f)=1$,
that are literals $x$ or $\bar{x}$. Therefore, in analogy to Section~\ref{sec:tautologies}, we will define so called simple $x$-trees.

\begin{df}
A simple $x$ is a tree of the shape
$x\land ST$, $x\lor SC$, $x\land(x\lor\cdots)$ or
$x\lor(x\land\cdots)$,
where $ST$ denotes a simple tautology and $SC$ a simple
contradiction. 
The shape of such trees is depicted in Figures~\ref{fig:simple_x} and ~\ref{fig:simple_x_asso}. 
\end{df}

\begin{figure}[htb]
\begin{center}
\includegraphics[width=0.8\textwidth]{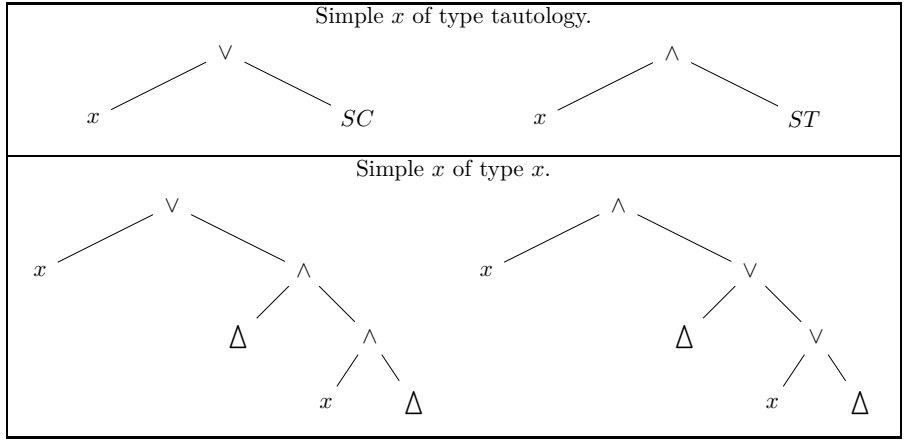}
\end{center}
\caption{The different kinds of simple $x$. Here, $ST$ is a simple
  tautology and $SC$ is a simple contradiction.}
\label{fig:simple_x}
\end{figure}

For all models, we will prove the following proposition:
\begin{prop}\label{prop:sx}
Asymptotically, almost all trees computing the function $x$ are simple $x$. 
\end{prop}

We state the proposition without a complete proof. The proof is easily done by similar arguments as in the previous section, using the patterns $N[N]$ or $R[R]$, respectively. We can prove that every tree $t\in \T$ or $t\in \C$ with exactly two $N[N]$-restrictions computing $x$, and every tree $t\in \A$ or $t\in\P$ with
exactly two $R[R]$-restrictions, respectively, computing $x$, is a simple $x$ tree.
Theorem~\ref{thm:Jakub} and its generalizations imply that those trees give asymptotically almost all trees computing $x$, as it is an easy task to prove that a large tree computing $x$ will have at least two restrictions.
Still we suggest a much simpler argument which proves the proposition in Section~\ref{sec:summary}.

\subsection{Binary plane trees.}

\begin{thm}
The limiting probability of functions of complexity $1$ in the binary plane model is 
\[\lim_{m\rightarrow \infty}\IP_{m,n}(x)=\frac{5}{16n^2}+\O\left(\frac{1}{n^3}\right).\]
\end{thm}

\begin{proof}
We distinguish between simple $x$ of type tautology, which we denote
by $x_T$,
and simple $x$ of type $x$, denoted by $x_X$
(c.f. Figure~\ref{fig:simple_x}). By Proposition~\ref{prop:sx},
we have $\IP_{n}(x)=\IP_{n}(x_T)+\IP_{n}(x_X)$. 

First we compute $\IP_{n}(x_T)=\lim_{m\to\infty}\IP_{m,n}(x_T)$.
Let $ST(z)$ be the generating function computing simple tautologies,
given in Section~\ref{sec:catalan-true}. Of course, $ST(z)$ also counts
contradictions. The generating function $\widetilde{ST}(z)$ of simple $x$
of the first kind is given by $4z\cdot ST(z)$, where the factor $z$
counts the leaf labelled with $x$, and the factor $4$ is explained by
the constant being a tautology or a contradiction, the label of the
internal node then being fixed, and the constant being positioned
left or right. Hence
\[\frac{[z^m]\widetilde{ST}(z)}{[z^m]T(z)}\sim4\rho_n\frac{[z^m]ST(z)}{[z^m]T(z)}=4\rho_n\IP_{n}(True)\sim\frac{3}{16n^2}.\]
For the computation of $\IP_{n}(x_X)$ we use the function $g_x(z)$
given in~\eqref{eq:catalangx}. Let $\tilde{g}_x(z)$ be the function
counting simple $x$ of type $x$. Then $\tilde{g}_x(z)=4zg_x(z)$
by the same arguments as above, hence
\[\frac{[z^m]\tilde{g}_x(z)}{[z^m]T(z)}\sim
4\rho_n\frac{[z^m]g_x(z)}{[z^m]T(z)}\sim\frac{4}{16n}
\lim_{z\to\frac{1}{16n}}\frac{g'(z)}{T'(z)}.\]
Using Maple, we get
$\lim_{z\to\frac{1}{16n}}\frac{g'(z)}{T'(z)}=\frac{1}{2n}+\O\left(\frac{1}{n^2}\right)$, hence 
\[\IP_{n}(x)=\IP_{n}(x_T)+\IP_{n}(x_X)=\frac{3}{16n^2}+\frac{1}{8n^2}+\O\left(\frac{1}{n^3}\right)=
\frac{5}{16n^2}+\O\left(\frac{1}{n^3}\right).\]
\end{proof}

\subsection{Associative plane trees.}

\begin{thm}\label{thm:assocx}
The limiting probability of functions of complexity $1$ in the associative model is 
\[\lim_{m\rightarrow \infty}\IP_{m,n}^a(x)=\frac{546-386\sqrt{2}}{n^2}+\O\left(\frac{1}{n^3}\right).\]
\end{thm}

\begin{proof}
Again we distinguish between simple $x$ of type tautology ($x_T$),
and simple $x$ of type $x$ ($x_X$, cf Figure \ref{fig:simple_x_asso}).
Note that a simple $x$ in the associative case is represented by a tree with a binary root.

\begin{figure}[htb]
\begin{center}
\includegraphics[width=0.4\textwidth]{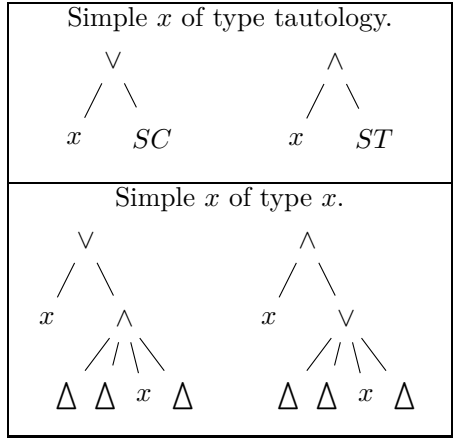}
\end{center}
\caption{The different kinds of simple $x$ in the associative case, up
  to commutativity. $ST$ is a simple tautology and $SC$ a simple
  contradiction.}
\label{fig:simple_x_asso}
\end{figure}

Calculating $\IP_{n}(x_T)=\lim_{m\to\infty}\IP_{m,n}(x_T)$, we obtain
$\widetilde{ST}(z)=4z\cdot ST(z)$ by the same arguments as above and
\[\frac{[z^m]\widetilde{ST}(z)}{[z^m]A(z)}\sim4\alpha_n\frac{[z^m]ST(z)}{[z^m]A(z)}=
 4\alpha_n\IP_{n}^a(True)\sim 4\,\frac{3-2\sqrt{2}}{2n}\,\frac{51-36\sqrt{2}}{n}=\frac{594-420\sqrt{2}}{n^2}.\]

The contribution of $x_X$ is counted by $\tilde{g}(z)=4z g(z)$,
where $g(z)$ counts trees with an $\lor$-root and exactly one leaf
labelled by $x$. Note that the other leaves may not be labelled with
$x$ neither with $\bar{x}$, because this would give a simple
tautology. Then, $g(z)$ is given by
\[g(z)=z\sum_{\ell\geq2}\ell(A(z)-2z)^{\ell-1}.\]
Maple computations give
$\lim_{z\to\frac{3-2\sqrt{2}}{2n}}\frac{g'(z)}{A'(z)}\sim\frac{3\sqrt{2}-4}{n}+\O\left(\frac{1}{n^2}\right)$, and thus
\[\frac{[z^m]\tilde{g}_x(z)}{[z^m]A(z)}\sim4\alpha_n\frac{[z^m]g_x(z)}{[z^m]A(z)}\sim
 4\,\frac{3-2\sqrt{2}}{2n}\,\frac{3\sqrt{2}-2}{n}=\frac{34\sqrt{2}-48}{n^2}+\O\left(\frac{1}{n^3}\right).\]
Adding the two limiting ratios gives the constant in Theorem~\ref{thm:assocx}. 
\end{proof}

\subsection{Binary commutative trees.}

\begin{thm}
The limiting ratio of functions of complexity~$1$ in the binary commutative model is 
\[\lim_{m\rightarrow \infty}\IP_{m,n}^c(x)=\frac{1153}{4096n^2}+\O\left(\frac{1}{n^3}\right).\]
\end{thm}

\begin{proof}
Simple $x$-trees are the same as in the plane binary case, but there
is no left-to-right order anymore.
Hence, $\tilde{ST}(z)=2\gamma_nST(z)$, and $ST(z)=C(z)-\bar{ST}(z)$
with $\bar{ST}(z)$ given in \eqref{eq:commG}. Hence 
\[\frac{[z^m]\tilde{ST}(z)}{[z^m]C(z)}\sim2\gamma_n\frac{[z^m]ST(z)}{[z^m]C(z)}=
2\gamma_n\IP_{n}^c(True)\sim
2\,\frac{1}{8n}\left(1+\frac1{8n}\right)\frac{641}{1024n}=
\frac{641}{4096n^2}+\O\left(\frac1{n^3}\right),\]
$\tilde{g}_x(z)=2\gamma_ng_x(z)$, and $g(z)=C(z)-\bar{g}_x(z)$ with
$\bar{g}(z)$ given in~\eqref{eq:commG} and
$\lim\frac{\bar{g}'_x(z)}{C'(z)}$
computed in the proof of Theorem~\ref{thm:commbinary}. Hence 
\[\frac{[z^m]\tilde{g}_x(z)}{[z^m]C(z)}
\sim 2\gamma_n\frac{[z^m]g_x(z)}{[z^m]C(z)}
\sim 2\,\frac{1}{8n}\left(1+\frac1{8n}\right)\frac{1}{2n}
=\frac{1}{8n^2}+\O\left(\frac1{n^3}\right)
=\frac{512}{4096n^2}+\O\left(\frac{1}{n^3}\right).\]
\end{proof}

\subsection{Associative and commutative trees.}

\begin{thm}
The limiting ratio of functions of complexity~$1$ in the associative and commutative model is 
\[\lim_{m\rightarrow \infty}\IP_{m,n}^{a,c}(x)=\frac{(2\ln2-1)^2(2\ln2+1)}{4n^2}+\O\left(\frac{1}{n^3}\right).\]
\end{thm}

\begin{proof}
Again, $\widetilde{ST}(z)=2\delta_nST(z)$, with $ST(z)=nST^x(z)$ and $ST^x(z)$ given in \eqref{eq:asscommG}, and 
\begin{align*}
\frac{[z^m]\widetilde{ST}(z)}{[z^m]P(z)} &\sim 2\delta_n\frac{[z^m]ST(z)}{[z^m]P(z)}=2\delta_n\IP_{m}^{a,c}(True)\\
& \sim2\frac{(2\ln2-1)}{2n}\frac{(2\ln2-1)^2}{8n}=\frac{(2\ln2-1)^3}{8n^2}+\O\left(\frac1{n^3}\right).
\end{align*}
Moreover $g_x(z)$ is given by
\[g_x(z)=z+z\left(\exp\left(\sum_{\ell\geq1}\frac{\hat{P}(z^\ell)-2z^\ell}{\ell}\right)-1\right),\]
and
\[g_x'(z)=1+\frac{1}{z}(g_x(z)-z)+g_x(z)\left(\sum_{\ell\geq1}z^{\ell-1}(\hat{P}(z^\ell)-2)\right).\]
Since $g_x(\rho)\sim2\rho$, we get
\[\lim_{z\to\delta_n}\frac{g_x'(z)}{P'(z)}\sim\lim_{z\to\delta_n}\frac{g_x(z)\hat{P}'(z)}{2\hat{P}'(z)-2n}\sim\frac{2\rho}{2}=\frac{2\ln2-1}{2n},\]
and finally, with $\tilde{g}_x(z)=2\delta_ng_x(z)$, 
\[\frac{[z^m]\tilde{g}_x(z)}{[z^m]P(z)}
\sim2\delta_n\frac{[z^m]g_x(z)}{[z^m]P(z)}
\sim2\,\frac{(2\ln2-1)}{2n}\,\frac{(2\ln2-1)}{4n}
=\frac{(2\ln2-1)^2}{4n^2}+\O\left(\frac1{n^3}\right).\]
\end{proof}

\section{Limiting probability of a general function.}\label{sec:gen-constants}

In the previous sections, we have studied functions of complexity zero and one. In this section we are interested in the limiting probability of functions of higher complexity. To prove Theorem~\ref{catalan-theta}, Kozik showed that asymptotically almost all trees computing a function $f$ have a ``simple~$f$'' shape.
To be more precise, they are obtained from a minimal tree by a single well-defined expansion, that
is a special tree attached to a node of a minimal one. In this section we generalise this result to all models
and give bounds for the number of such expansions.

\subsection{The binary plane case.}

The goal of this section is to prove existence and bound the constant $\lambda_f$ appearing in Theorem~\ref{catalan-theta}. We show the following result. 
\begin{prop}
\label{lambda-catalan}
For all Boolean functions $f$,
\[\frac{8L(f)-3+\ell}{16^{L(f)}}\, M_f\leq \lambda_f \leq
\frac{4L(f)^2+4L(f)-3}{16^{L(f)}}\, M_f\]
where $M_f$ is the number of minimal trees representing~$f$ and $\ell=\lceil\frac{L(f)}{2}\rceil$ for $L(f)>1$ and $0$ for $L(f)=1$.
\end{prop}

The proof of this proposition is based on a result by Kozik~\cite{Kozik}. He proved that the set of non negligible trees computing~$f$ is exactly the set of trees obtained by \emph{expanding} a minimal tree of~$f$ once.

\begin{rmq}
It is interesting to see that these bounds are equal when the complexity of the function is~1 and give the actual bound for literals we computed in Section~\ref{sec:literals}.
\end{rmq}

\begin{df}
Let $t$ be an And/Or tree computing $f$, $\nu$ one of its nodes
and $t_\nu$ the subtree rooted at $\nu$.
An \emph{expansion} of $t$ in $\nu$ is a tree obtained by
replacing the subtree $t_\nu$ rooted at $\nu$ by a tree
$t_\nu \diamond t_e$ where $\diamond \in \{\land,\lor\}$
and where $t_e$ is an And/Or tree.
We will say that such an expansion is \emph{valid} when
the expanded tree still computes~$f$.
\end{df}

Kozik has shown that the only non-negligible valid expansions
that are to be considered are:
\begin{itemize}
\item The T-expansions: a valid expansion is a T-expansion if the inserted
subtree $t_e$ is a simple tautology (resp. a simple contradiction) and
if the new label of $\nu$ is $\land$ (resp. $\lor$).
\item The X-expansions: a valid expansion is an X-expansion if the inserted
subtree $t_e$ is (up to commutativity and associativity) of the shape
$x\lor...$ (resp. $x\land...$) where $x$ is an essential variable
of~$f$ and if the new label of $\nu$ is $\land$ (resp. $\lor$).
\end{itemize}

\begin{rmq}
It is important to note that all these expansions are valid.
\end{rmq}

In the following, we will name a T-tautology expansion an $\land$-T-expansion (resp. an $\lor$-T-expansion) if the new label of $\nu$ is $\land$ (resp. $\lor$), and the same for X-expansions.

\begin{proof}[Proof of Proposition \ref{lambda-catalan}]
In a Catalan And/Or tree, a T-expansion is possible in every node (without changing the computed function). At each node, we can expand by an $\lor$-T-expansion and by an $\land$-T-expansion, both on the right side and on the left side. As a minimal tree of $f$ is of size $L(f)$, it has $2L(f)-1$
nodes and there are $\lambda_T(f)=4(2L(f)-1)M_f$ different T-expansions that can be done in all minimal trees computing~$f$.

We can now consider $\lambda_X(f)$, that is, the number of X-expansions (which do not change the computed function~$f$). This number depends heavily on the shape of the minimal trees of~$f$, therefore, we only give bounds for this number. An $\land$-X-expansion (resp. $\lor$-X-expansion) realized by $x_i$ is valid at each leaf labelled by $x_i$, as well as at each node connected to one of them by an $\lor$-only (resp. $\land$-only) path, and at all of its sons. Let us note that:
\begin{itemize}
\item[-] At each leaf, we can do at least one $\lor$-X-expansion and one $\land$-X-expansion, both to the
right and to the left, which gives a contribution of $4L(f)$. Further, we could do either one
$\land$-X-expansion or one $\lor$-expansion to both sides at its father, depending on its level.
But two different leaves having the same father could be labelled by the same variable. Hence this contributes $2\lceil\nicefrac{L(f)}{2}\rceil$ if $L(f)>1$, as if $L(f)=1$ a minimal tree consists of a single leaf with no father, but else all leaves share its father with one leaf of the same label in the worst case.

\item[-] at each node (internal or external), we can do at most 4 $X$-expansions (we choose between $\land$ and $\lor$ and between right and left side) for each different literal that appear on the leaves. There are at most $L(f)$ different literals appearing on the leaves of a minimal tree and a minimal tree has exactly $2L(f)-1$ (internal or external) nodes. Therefore, $4L(f)(2L(f)-1)M_f$ is an
upper bound of $\lambda_X(f)$.
\end{itemize}

Therefore, we have the following bounds:
\begin{equation}\label{bounds-catalan}
5L(f)M_f \leq \lambda_X(f) \leq 4L(f)(2L(f)-1)M_f.
\end{equation}

To end the proof of Proposition \ref{lambda-catalan},
we only need to note that:
\[\frac{\lambda_f}{n^{L(f)+1}} =
M_f \rho_n^{L(f)}\left(\lambda_T(f) w_1+\lambda_X(f) w_2\right),\]
where $w_1$ is the limiting ratio of simple tautologies
(resp. simple contradiction), and $w_2$ is the limiting ratio
of trees of shape $x\lor...$ for $x$ a variable. Thanks to
computations made in Section~\ref{sec:tautologies}
(c.f. Theorem~\ref{catalan-constant}), we know that
$w_1=\frac{3}{4n}$.
Moreover, the generating function $g_x$ defined in
Section~\ref{sec:catalan-true} counts exactly the number of trees
that can be used for an X-expansion (according to a variable~$x$).
Therefore,
\[\lim_{z\rightarrow \rho_n}\frac{g_x'(z)}{T'(z)}\sim \frac{1}{2n} = w_2,\]
and with~\eqref{bounds-catalan} we prove Proposition~\ref{lambda-catalan}.
\end{proof}

\subsection{The associative plane case.}\label{sec:generalass}

The associative case appears to be similar to the binary plane case. We prove the following theorem:

\begin{thm}\label{lambda-assoc}
In the associative plane model, let $f$ be a non-constant Boolean function, whose complexity is denoted by $L(f)$. Then
\[\IP^a_n(f)\sim \frac{\lambda^a_f}{n^{L(f)+1}},\]
when $n$ tends to infinity, where $\lambda^a_f$ is depending on the number of possible expansions
of minimal trees of~$f$. For $L(f)>1$ we have \[\left(\frac{3-2\sqrt{2}}{2}\right)^{L(f)}\left[133L(f)+153-(93L(f)+108)\sqrt{2}\right]M_f \leq \lambda_f\]
\[\lambda_f \leq \left(\frac{3-2\sqrt{2}}{2}\right)^{L(f)}\left[-(12L(f)^2-247L(f)+51)+(9L(f)^2-174L(f)+36)\sqrt{2}\right]M_f,\]
where $M_f$ is the number of minimal trees computing~$f$.
\end{thm}

To prove Theorem \ref{lambda-assoc}, we first have to prove that, as in the binary plane case,
the set of non negligible associative trees computing a Boolean function is the set of trees
obtained from a minimal tree by expanding it once. Moreover, we have to find the non-negligible
expansions that have to be considered. Then, we can prove Theorem~\ref{lambda-assoc}
with the same methods as in the binary plane case.

\subsubsection{Associative expansions.}

Because of the stratified structure of associative trees, we have to be careful with the definition of expansions, which is different to the one in the binary case:

\begin{figure}[htb]
\begin{center}
\includegraphics[width=0.6\textwidth]{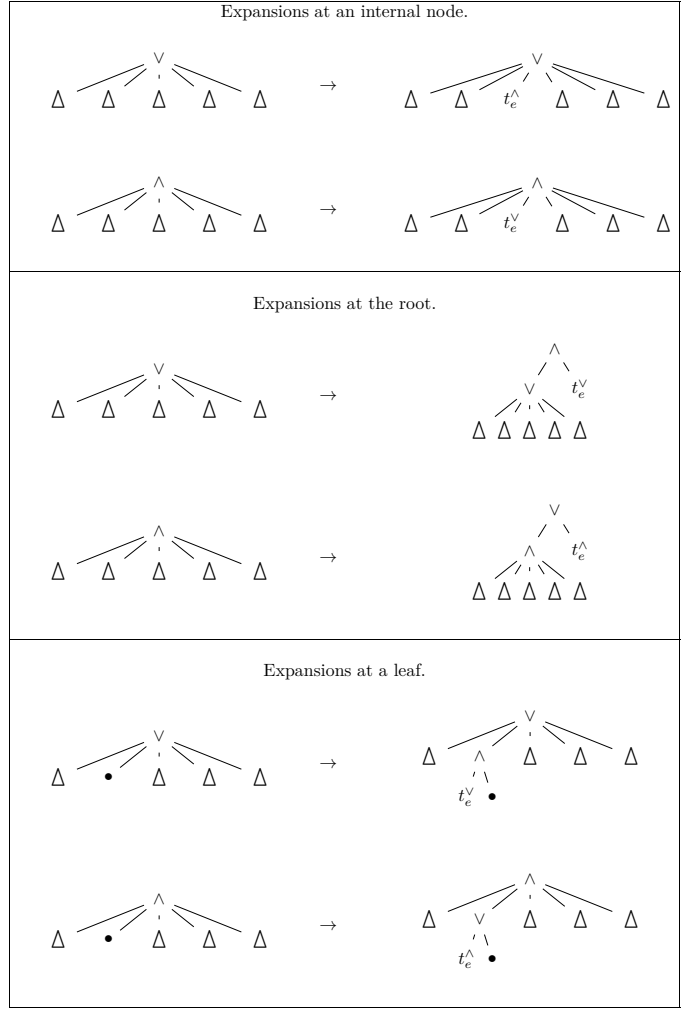}
\end{center}
\caption{The different possible expansions in the associative case. Here, $t_e^\lor$ (resp. $t_e^\land$) stands for an associative tree rooted by $\lor$ (resp. $\land$). The tree pictured is only the subtree rooted at $\nu$, before the expansion and after the expansion.}
\label{expansions-assoc}
\end{figure}

\begin{df}[c.f. Figure \ref{expansions-assoc}]
Let $t$ be an And/Or associative tree computing $f$. We define two types of expansions of $t$. 
\begin{itemize}
\item Let $\nu$ be an internal node of $t$ (possibly the root) with subtrees $t_1,\ldots,t_j$, $j\geq 2$.
An expansion of the first kind of $t$ in $\nu$ is a tree obtained by adding a subtree $t_e$ to $\nu$.
\item Let $\nu$ be the root or a leaf of the tree. The tree obtained by replacing the subtree $t_\nu$
rooted at $\nu$ by $t_e \diamond t_\nu$, where $\diamond\in \{\land,\lor\}$ is chosen such that the
obtained tree is stratified, is an expansion of $t$ in $\nu$ of the second kind.
In this case, $\diamond$ will be called the \emph{new label} of $\nu$.
\end{itemize}
We will say that such an expansion is \emph{valid} when
the expanded tree still computes~$f$.
\end{df}

\begin{rmq}
We have to keep in mind that both types of expansions are possible at the root of the tree.
\end{rmq}

\begin{prop}
The set of non-negligible trees computing a Boolean function $f$ is the set of trees obtained by expanding
a minimal tree of $f$ once. Moreover, the only non-negligible valid expansions are:
\begin{itemize}
\item The T-expansions: a valid expansion is a T-expansion if the inserted subtree $t_e$ is a simple tautology
(resp. a simple contradiction) and if the new label of $\nu$ is $\land$ (resp. $\lor$).
\item The X-expansions: a valid expansion is an X-expansion if the inserted subtree $t_e$ is
(up to commutativity) of the shape $x\lor\dots$ (resp. $x\land\dots$) where $x$ is an essential variable of $f$ and if the label of the father of $t_e$ is $\land$ (resp. $\lor$).
\end{itemize}
\end{prop}

Before proving the result, let us introduce the pattern language we will need:
\begin{equation}
\left\{
\begin{array}{l}
\hat{P} = \bullet | \check{P} \land \check{P} | \check{P} \land \check{P} \land \check{P} | \ldots \\
\check{P} = \bullet | \hat{P} \lor \boxempty | \hat{P} \lor \boxempty \lor \boxempty | \ldots \\
S =\{\hat{P},\check{P}\};\\
\hat{N} = \bullet | \check{N}\land \boxempty | \check{N}\land \boxempty
\land \boxempty | \ldots \\
\check{N} = \bullet | \hat{N}\lor \hat{N} | \hat{N}\lor \hat{N} \lor
\hat{N} | \ldots \\
R = \{\hat{N},\check{N}\}.
\end{array}
\right.
\end{equation}

\begin{rmq}
\begin{itemize}
\item[-] The pattern $S$ has the following property: if all $S$-pattern leaves of a tree are set to $True$, then the whole tree itself computes $True$.
\item[-] The pattern $R$ has the following property: if all $R$-pattern leaves of a tree are set to $False$, then the whole tree itself computes $False$.
\end{itemize}
\end{rmq}

\begin{df}
Let us consider the pattern $L^{r}$ and $\bar{L}=L^{r+1}$. For $i\leq r$, a leaf is \emph{on level $i$} if it is a $L^{(i)}$-pattern leaf but not a $L^{(i-1)}$-pattern leaf. A $\bar{L}$-pattern leaf which is not a $L$-pattern leaf is \emph{on level $r+2$}. 
\end{df}

\begin{proof}
The proof is inspired by the proof of the corresponding result for binary trees (see~\cite[Theorem 6.1]{Kozik}). The idea is to take a tree computing~$f$, and to replace every subtree which can be evaluated
to $True$ or $False$ independently from the rest of the tree by a $\star$.
Then, we state that simplifying the stars gives a minimal tree of the considered Boolean function.\\

Let $f$ be a Boolean function whose complexity is $L(f)$. We consider the pattern languages $L=R^{(L(f)+1)}[R\oplus S]$ and $\bar{L}=R^{(L(f)+1)}[(R\oplus S)^2]$, where the pattern leaves of $R\oplus S$ are all pattern leaves of the pattern $R$ and the pattern $S$. 

Let $t$ be a tree of size $L(f)$ representing $f$. If the root of $t$ is labelled with $\lor$ (resp. $\land$), then using a simple contradiction (resp. tautology) $\Phi$, the new tree $\Phi \land t$ (resp. $\Phi \lor t$) still represents the function $f$. Since the limiting ratio of simple tautologies or contradictions is equal to $\Theta(1/n)$ and the $L(f)$ nodes of $t$ are counted by $z^{L(f)}$, for large enough $n$ we obtain the following lower bound:
$$\IP^a_n(f)\geq \frac{\alpha}{n^{L(f)+1}}.$$
Consequently, we can neglect trees with at least $L(f)+2$ $\bar{L}$-restrictions.

Further, we can prove that trees with strictly less than $L(f)+1$ $L$-restrictions cannot represent $f$, because we would then be able to simplify this tree to a tree computing $f$ of size smaller than $L(f)$.

Finally, we know that the set of non negligible trees computing $f$ is the set of trees with exactly $L(f)+1$ $L$-restrictions and $L(f)+1$ $\bar{L}$-restrictions. Thus, we know that every variable appearing in a pattern leaf on level $L(f)+3$ is non essential and not repeated among the $\bar{L}$ pattern leaves. Therefore, each subtree of $t$ rooted on level $L(f)+3$ with its parent node on level $L(f)+2$\footnote{i.e. the parent node of the root of $t$ has been a $\boxempty$-vertex in the $L(f)+2^{nd}$ step} can be replaced by a $\star$, because it can be valuated to $False$ by assigning all the $R$-pattern leaves to $False$ or to $True$ by assigning all the $S$-pattern leaves to $True$. Both valuations can be done independently
from the rest of the tree because these pattern leaves are non essential and not repeated.
After this operation all the remaining leaves are $\bar{L}$-pattern leaves.

Moreover, we replace every leaf of the tree by a $\star$ which is not an essential variable of~$f$
and which appears only once among the leaves of $t$. We have obtained a tree $t^\star$.

We can simplify the tree in order to obtain a tree without stars according to the following rules:
\begin{center}
\begin{align}
\nonumber 	\star \lor \ldots \lor \star \equiv \star \quad
			&\quad \star \land \ldots \land \star \equiv \star\\
\label{simpl} \star \lor \ldots \lor t_1 \lor \ldots \lor t_j \equiv True \quad
			& \quad \star \land \ldots \land t_1 \land \ldots \land t_j \equiv False \\
\nonumber 	True \lor t_1 \lor \ldots \lor t_j \equiv True \quad
			& \quad False \land t_1 \land \ldots \land t_j \equiv False \\
\label{last}	False \lor t_1 \lor \ldots \lor t_j \equiv t_1 \lor \ldots \lor t_j \quad
		 	& \quad True \land t_1 \land \ldots \land t_j \equiv t_1 \land \ldots \land t_j
\end{align}
\end{center}
where $t_1,\ldots,t_j$ are subtrees containing no stars.

The tree $\hat{t}$ obtained after this process still computes $f$.
Let us prove that the obtained tree is a minimal tree of $f$ and that
the least common ancestor\footnote{the least common ancestor of a set of nodes is the node $\nu$ farthest from the root
such that the tree rooted at $\nu$ contains all the nodes of the considered set.}
of the stars in $t^\star$ has been simplified during the process.

The tree $t^\star$ contains at least one $\star$ since $t$ is big enough to have
at least one leaf on level $L(f)+3$. Moreover, the tree $t^\star$ has exactly $L(f)+1$ restrictions and no constants.
The final tree $\hat{t}$ has no $\star$, and no constant. Therefore,
a rule of type \eqref{simpl} must have been used at least once during the simplification process,
because the rules \eqref{simpl} are the only ones simplifying stars.
But, using such a rule simplifies at least one subtree with at least one non-star pattern leaf.
Therefore, this leaf had to be labelled by a variable which was either essential or repeated.
Therefore, the simplifying process has at least simplified one restriction and the obtained tree has at most $L(f)$ leaves.
Thus, $\hat{t}$ is a minimal tree of $f$.

Moreover, let $\nu$ be the least common ancestor of the stars in $t^\star$.
Let us assume that $\nu$ does not disappear during the simplification process.
Therefore, two stars have been simplified independently during the process, and thus
at least two rules of type \eqref{simpl} have been applied, which means that at least
two restrictions have disappeared during the simplification process. Therefore, the simplified
tree $\hat{t}$ still computes $f$ and contains at most $L(f)-1$ restrictions, i.e. at most $L(f)-1$ leaves,
which is impossible since the complexity of $f$ is $L(f)$.
Therefore, the node $\nu$ has dissappeared during the simplification process.

Finally, let us remark that the last simplifying rule applied during the process has to be of the kind \eqref{last}. We have thus shown that a non-negligible tree computing $f$ is indeed a minimal tree expanded once.

The second part of the proof is to understand which are the non-negligible valid expansions, i.e. those which do not change the computed function $f$.

First, let us remark that, thanks to Theorem~\ref{thm:Jakub} and its generalizations, the trees obtained by expanding with a tree $t_e$ with more than two $(R \oplus S)^2$ restrictions are negligible. On the other hand there has
to be at least one $(R \oplus S)^2$ restriction in $t_e$, because if there was
none, we could assign this tree to $False$ or $True$ independently from the rest of the tree.
Since the expanded tree must still compute the function $f$, by simplification we would obtain
a tree computing $f$ being smaller than the minimal tree, which is impossible.

\paragraph*{First case:} The tree $t_e$ contains one repetition and no essential variable of $f$.
Then, it has to compute a constant function (i.e. $True$ or $False$). If it does not, by previous arguments on tautologies, the subtree can be valuated to $True$ or $False$ independently from the rest of the tree.
Thus, by simplification, we can obtain a tree, smaller than the minimal tree, computing $f$,
which is a contradiction. Therefore, the expanding tree $t_e$ is a simple tautology or a
simple contradiction (thanks to Proposition~\ref{st-assoc}).
Moreover, as the expanded tree still has to compute $f$,
if the father of $t_e$ is an $\land$ (resp. $\lor$), $t_e$ is a simple tautology
(resp. contradiction), which gives a $T$-expansion. 

\paragraph*{Second case:} The subtree $t_e$ contains no repetition and one essential variable,
let us say $x$. Then, the essential variable has to appear on the first level.
If it does not, the Boolean expression has shape
$s_1 \land (s_2 \lor x)$ or $s_1 \lor (s_2 \land x)$ (up to commutativity).
Moreover, the trees $s_1$ and $s_2$ have no $R \oplus S$-restrictions and therefore
we can make them $False$ or $True$ independently from the rest of the tree.
Then, we can valuate the whole tree either to $False$ or $True$ independently from $x$,
which is impossible since $x$ is an essential variable of $f$.

If an $\land$-X-expansion $t_e$ according to the variable $x_i$ is valid in a node $\nu$,
then every $\land$-X-expansion $t'_e$ according to this variable $x$ is valid
at $\nu$ (and as well for $\lor$-X-expansions).
\end{proof}

\subsubsection{Computing bounds for $\lambda_f$.}

\begin{proof}
As in the binary case, we have to compute the limiting ratio of T-expansions and X-expansions, and the number of nodes where the different kinds of expansions are allowed. Let us denote by $M_f$ the number of minimal trees representing a given Boolean function~$f$ of complexity~$L(f)$.

The limiting ratio of $T$-expansions is the limiting ratio of
simple tautologies, which has already been computed in
Section~\ref{sec:assoc}. We have that $w_1^a=\frac{51-36\sqrt{2}}{n}$.

Let $g_x$ be the generating function of associative trees
rooted at $\land$ (resp. $\lor$) and containing exactly
one~$x$ in the first generation. Then,
\[g_x(z)=z\sum_{j\geq 2} j (A(z)-2z)^{j-1}.\]
Since the set of trees with more than one~$x$ in the first
generation is negligible in front of the set of trees with
exactly one~$x$ in the first generation, we can assume that:
\[w_2^a=\lim_{z\rightarrow \alpha_n} \frac{g'_x(z)}{A'(z)}=\frac{3\sqrt{2}-4}{n}\]
is the limiting ratio of $\land$-X-expansions (resp. $\lor$-X-expansions).

As in the binary case, the number $\lambda_X(f)$ of $X$-expansions and the number $\lambda_T(f)$ of $T$-expansions allowed in a minimal tree depend on the shape of the considered minimal tree. Given a minimal tree $t$ of $f$, let us number its internal nodes from $1$ to $N$. Let us denote by $s(i)$ the number of sons of the internal node~$i$. Moreover, let us denote by $d(i)$ the number of sons of the node~$i$ which are leaves. Then, if $\lambda_T(t)$ is the number of different $T$-expansions in the minimal tree $t$ of $f$, we have that:
\[\lambda_T(t) = 2L(f) + \sum_{i=1}^N (s(i)+1) + 2,\]
where $2L(f)$ is the number of different $T$-expansions allowed
at the leaves of the tree (if the parent node is labelled by $\land$
(resp. $\lor$), only simple tautology (or contradiction respectively)
T-expansions are allowed), $s(i)+1$ is the number of different
$T$-expansions allowed at the node $i$ (the number of different
positions at node $i$ is $s(i)+1$); and $2$ is the number of
expansions
allowed at the root by pushing the root to the first generation and
adding a new root with two sons. Therefore,

\[\lambda_T(t) = 2L(f)+ \sum_{i=1}^N s(i) + N +2 = 2L(f) + (L(f)+N-1) +N +2,\]

and since $1\leq N \leq L(f)-1$, we obtain that:

\[3(L(f)+1)M_f \leq \lambda_T(f) \leq (5L(f)-1)M_f.\]

Further, given a leaf $x_i$, an $\land$-X-expansion realized by~$x_i$ is allowed at itself, at its father and at all its sisters (brothers that are reduced to a leaf), because two sisters cannot have the
same label. Indeed, if two sisters have the same label (or even opposite labels), then, the considered tree can be simplified, and since we consider a minimal tree, this is impossible. Therefore,  if $L(f)>1$,
\[\lambda_X(t) = \sum_{i=1}^N d(i)(s(i)+1) + 2d(root) + 2\sum_{i=1}^N d(i)^2.\]

\begin{lem}
For all $i$, $d(i)\leq L(f)-N+1$.
\end{lem}
\begin{proof}
Let us assume that there exist an internal node $i_0$ such that
$d(i_0)>L(f)-N+1$.
It is easy to see that, as each node except the root has a unique
father, $\sum_{i=1}^N d(i)=L(f)+N-1$.
Moreover,
\begin{equation*}
\sum_{i=1}^N d(i)>\sum_{i\neq i_0}d(i)+(L(f)-N+1)  > 2(N-1)+(L(f)-N+1)
\end{equation*}
since every internal node has at least two sons.
Therefore, $\sum_{i=1}^N d(i)>L(f)+N-1$, which is a contradiction.
\end{proof}
Therefore, thanks to the lemma,
\begin{align*}
\lambda_X(t) &\leq \sum_{i=1}^N d(i) + (L(f)-N+1)\sum_{i=1}^N s(i) + 2(L(f)-N+1) + 2(L(f)-N+1)\sum_{i=1}^N d(i)\\
&\leq L(f)+ (L(f)-N+1) [(N+L(f)-1) + 2 + 2L(f)]\\
&\leq L(f)\cdot (3L(f)+2).
\end{align*}
On the other hand,
\[\lambda_X(t) \geq 3\sum_{i=1}^N d(i) + 2\sum_{i=1}^N d(i) = 5 L(f),\]
and
\[\lambda_X(f)\geq 5 L(f) M_f.\]
Finally, since
\[\frac{\lambda_f}{n^{L(f)+1}}=M_f \alpha_n^{L(f)} (\lambda_T(f) w_1^a + \lambda_X(f) w_2^a),\]
we get that:
\[\left(\frac{3-2\sqrt{2}}{2}\right)^{L(f)}\left[133L(f)+153-(93L(f)+108)\sqrt{2}\right]M_f \leq \lambda_f\]
\[\lambda_f \leq
\left(\frac{3-2\sqrt{2}}{2}\right)^{L(f)}\left[-(12L(f)^2-247L(f)+51)+(9L(f)^2-174L(f)+36)\sqrt{2}\right]M_f.\]
\end{proof}

\subsection{The binary commutative case.}\label{sec:generalcomm}

\begin{thm}
In the binary commutative case, let $f$ be a non-constant Boolean function, whose complexity is denoted by $L(f)$. Then
\[\IP^c_n(f)\sim \frac{\lambda^c_f}{n^{L(f)+1}},\]
when $n$ tends to infinity, where $\lambda^c_f$ is depending on the number of possible expansions
of minimal trees of~$f$, and
\[\frac{1794L(f)-641}{512\cdot 8^{L(f)}} M_f \leq \lambda^c_f \leq \frac{(2L(f)-1)(512L(f)+641)}{512\cdot 8^{L(f)}} M_f,\]
where $M_f$ is the number of minimal trees computing~$f$.
\end{thm}

\begin{rmq}
It is interesting to see that these bounds are equal when the complexity of the function is~1 and give the limiting probability of literals computed in Section~\ref{sec:literals}. 
\end{rmq}

\begin{proof}
The proof relies completely on the binary plane case, doing minimal $[N]$-embeddings and $[N\oplus P]$-embeddings (the plane parts of an $[N\oplus P]$-embeddings are both the plane parts of an $[N]$-embedding or a $[P]$-embedding). It has been proven in Lemmas~\ref{lem:power} and~\ref{lem:subcritical} that Theorem~\ref{thm:Jakub-comm-bin} can be applied to the pattern $[N\oplus P][\C]$, as the generating function of $P$ is the same as the one of $N$. As in the proof of
Theorem~\ref{thm:commbinary}, embedding a tree $t\in\C$ into $N^{(L(f))}[N\oplus P]$ or
$N^{(L(f))}[N\oplus P]^{(2)}$ represents an injection. Hence asymptotically, all trees computing
a function~$f$ are obtained by a single expansion of a minimal tree of~$f$.
The calculation of the bounds can be done in the same way as in
the plane binary case. If we denote by $w_1^c$ the limiting ratio of
simple tautologies and by $w_2^c$ the limiting ratio of
$X$-expansions.
From Section~\ref{sec:comm-tautology} we know that $w_1^c=\frac{641}{1024n}$, and from Section~\ref{sec:literals}
we know that $w_2^c=\frac1{2n}$. Moreover, since asymptotically all trees computing$f$ are obtained by a  single expansion of a minimal tree, we have $\frac{\lambda_f^c}{n^{L(f)+1}} =\gamma_n^{L(f)} (\lambda_T w_1^c+ \lambda_X w_2^c)$ and
\begin{align*}
\lambda_T &= 2(2L(f)-1)M_f\\
2L(f) M_f \leq \lambda_X &\leq 2L(f)(2L(f)-1) M_f,
\end{align*}
since $\gamma_n\sim \frac1{8n}$ when $n\to \infty$, the result follows.
\end{proof}

\subsection{The associative commutative case.}

\begin{thm}
In the associative and commutative model,
let $f$ be a non-constant Boolean function, whose complexity is denoted by $L(f)$:
\[\IP_n^{a,c}(f) \sim \frac{\lambda^{a,c}_f}{n^{L(f)+1}},\]
when $n$ tends to infinity, where $\lambda^{a,c}_f$
is depending on the number of possible expansions
of minimal trees of~$f$. For $L(f)>1$
\[\left(\frac{2\ln 2-1}{2}\right)^{L(f)}\left(\left(\ln^2 2-\frac14\right)L(f)+\ln^22-2\ln2+\frac12\right)M_f\leq \lambda^{a,c}_f\]
\[\lambda^{a,c}_f \leq \left(\frac{2\ln 2-1}{2}\right)^{L(f)}\frac{(2\ln 2-1)(L(f)+1+4\ln 2)L(f)}{4}M_f,\]
where $M_f$ is the number of minimal trees computing~$f$.
\end{thm}

\begin{proof}
The result is easily proven by using the pattern $R^{(L(f))}[R\oplus S]$
and applying arguments of Sections ~\ref{sec:generalass} and~\ref{sec:generalcomm}.
Therefore, we have $\frac{\lambda_f^{a,c}}{n^{L(f)+1}}=\delta_n^{L(f)}(\lambda_T w_1^{a,c}+\lambda_X w_2^{a,c})$. The calculation of the bounds is similar to the computations done in the plane associative case.
From Section~\ref{sec:ass-comm-tautology} we know that $w_1^{a,c}=\frac{(2\ln 2 -1)^2}{4n}$
and in Section~\ref{sec:literals} we obtained $w_2^{a,c}=\frac{2\ln 2-1}{4n}$.
Moreover, $\delta_n \sim \frac{2\ln 2-1}{2n}$. We can show
\begin{align*}
(L(f)+2)M_f&\leq \lambda_T \leq 2L(f)M_f\\
2L(f)M_f&\leq \lambda_X\leq (L(f)^2+3L(f))M_f,
\end{align*}
where the lower bound holds only for $L(f)>1$, and the theorem is proven.
\end{proof}

\section{Summary of results and conclusion.}\label{sec:summary}

Finally, we have, in three steps, understood better the influence of associativity and commutativity on the behaviour of the limiting ditribution on Boolean functions induced by their tree representation. Indeed, we have shown that associativity and commutativity do not change the order of $\IP_{n}(f)$
when $n$ tends to infinity, it is still of order $\Theta\left(n^{-(L(f)+1)}\right)$.
But, in Sections~\ref{sec:tautologies} and~\ref{sec:literals} we showed that associativity and commutativity change the exact limiting ratio of tautologies and literals. We summarize the different constants computed in Figure~\ref{tab:constants}.
Section~\ref{sec:gen-constants} gives bounds for the constants for a general function and gives an interesting result about the shape of an average tree computing a fixed function: it is a minimal tree expanded once. The proofs of most of our results are based on the theory of patterns introduced by Kozik in~\cite{Kozik}.
We have generalised this theory to associative and commutative trees to show most of the results still hold in more general models.
\begin{figure}[htb]
\begin{center}
\begin{tabular}{|c|cccc|}
\hline
 & Catalan & Associative & Commutative & General\\
 & trees & (non-binary) & (non plane) & trees\\
 & & trees & trees & \\
\hline
 & & & & \\
$True$ 	& $\displaystyle\frac{3}{4} = 0.75$ &
$51-36\sqrt{2}\approx 0.0883$ & $\displaystyle\frac{641}{1024}\approx 0.626$ & $\displaystyle\frac{(2\ln 2 -1)^2}{4} \approx 0.0373$ \\
 & & & & \\
$x$	& $\displaystyle\frac{5}{16}\approx 0.313$ &
$546-386\sqrt{2}\approx 0.114$ &
$\displaystyle\frac{1153}{4096}\approx 0.281$ & $\displaystyle\frac{(2\ln 2 -1)^2(2\ln 2+1)}{4}\approx 0.0890$\\
 & & & & \\
\hline 
\end{tabular}
\end{center}
\caption{The different constants $\lambda$ such that $\IP(True)\sim
\frac{\lambda}{n}$ and $\IP(x)\sim \frac{\lambda}{n^2}$
when $n$ tends to infinity, depending on the studied model of trees.}
\label{tab:constants}
\end{figure}

Note that the simple $x$ trees we defined in Section~\ref{sec:literals} are exactly those trees obtained by expanding once a tree consisting of a single leaf $x$. Hence the proof of Proposition \ref{prop:sx} is immediate. 

We should also note that the relation between probability and complexity of a Boolean function holds for a fixed function $f$. It is not valid uniformly over \emph{all} Boolean functions.
Let us recall what the \emph{Shannon effect} is: If we choose a
Boolean function on $n$ variables \emph{uniformly} at random, asymptotically almost surely the function has a complexity which is exponential in $n$. In our models, we are still unable to compute the average complexity of a Boolean function.
Further work (similar to~\cite{GG10} on implicational logic) is required before (most probably) invalidating the Shannon effect for these non-uniform probability distributions.\\

%%%%%%%%%%%%%%%%%%%%%%%%%%%%%%%%%%%%%%%%%%%%%%

\bibliographystyle{abbrv}
\bibliography{boolean}

%%%%%%%%%%%%%%%%%%%%%%%%%%%%%%%%%%%%%%%%%%%%%%

\end{document}